\documentclass{amsart}

\title    {On indicators of Hopf algebras}
\author   [K. Shimizu]{Kenichi Shimizu}
\date     {}
\address  {Graduate School of Mathematics, Nagoya University, Furo-cho, Chikusa-ku, Nagoya 464-8602, Japan}
\email    {x12005i@math.nagoya-u.ac.jp}


\usepackage{amsmath, amssymb, amscd, amsthm}

\numberwithin{equation}{section}

\newtheorem{counter}            {}[section]
\theoremstyle{definition}

\theoremstyle{plain}
\newtheorem{lemma}              [counter]{Lemma}
\newtheorem{proposition}        [counter]{Proposition}
\newtheorem{theorem}            [counter]{Theorem}
\newtheorem{corollary}          [counter]{Corollary}
\newtheorem{question}           [counter]{Question}

\theoremstyle{remark}
\newtheorem{remark}             [counter]{Remark}
\newtheorem{example}            [counter]{Example}

\newtheorem*{note*}{Note}

\DeclareMathOperator{\op}       {op}
\DeclareMathOperator{\cop}      {cop}
\DeclareMathOperator{\gr}       {gr}
\DeclareMathOperator{\id}       {id}
\DeclareMathOperator{\Hom}      {Hom}
\DeclareMathOperator{\End}      {End}

\DeclareMathOperator{\qexp}     {qexp}
\DeclareMathOperator{\Trace}    {Tr}
\DeclareMathOperator{\Lw}       {Lw}
\DeclareMathOperator{\ord}      {ord}
\DeclareMathOperator{\lcm}      {lcm}
\DeclareMathOperator{\corad}    {\mathrm{corad}}

\newcommand{\FiltVect} {\mathsf{FiltVec}}
\newcommand{\GrVect}   {\mathsf{GrVec}}
\newcommand{\YD}       {\mathcal{YD}}
\newcommand{\Mod}  [1] {{}_{#1}\mathcal{M}}
\newcommand{\fdMod}[1] {{}_{#1}\mathcal{M}_\mathit{fd}}

\begin{document}

\maketitle

\begin{abstract}
  Kashina, Montgomery and Ng introduced the $n$-th indicator $\nu_n(H)$ of a finite-dimensional Hopf algebra $H$ and showed that the indicators have some interesting properties such as the gauge invariance. The aim of this paper is to investigate the properties of $\nu_n$'s. In particular, we obtain the cyclotomic integrality of $\nu_n$ and a formula for $\nu_n$ of the Drinfeld double. Our results are applied to the finite-dimensional pointed Hopf algebra $u(\mathcal{D}, \lambda, \mu)$ introduced by Andruskiewitsch and Schneider. As an application, we obtain the second indicator of $u_q(\mathfrak{sl}_2)$ and show that if $p$ and $q$ are roots of unity of the same order, then $u_p(\mathfrak{sl}_2)$ and $u_q(\mathfrak{sl}_2)$ are gauge equivalent if and only if $q = p$, where $p$ and $q$ are roots of unity of the same odd order.
\end{abstract}

\tableofcontents

\section{Introduction}

Roughly speaking, a monoidal category is a category endowed with an associative and unital binary operation between its objects. An example is the category $\Mod{H}$ of modules over a Hopf algebra $H$ over a field $k$. Two Hopf algebras $H$ and $H'$ over $k$ are said to be {\em monoidally Morita equivalent} if $\Mod{H}$ and $\Mod{H'}$ are equivalent as $k$-linear monoidal categories. From the view point of the representation theory of Hopf algebras, it is natural and important to study this equivalence relation.

By the result of Schauenburg \cite{MR1408508}, two finite-\hspace{0pt}dimensional Hopf algebras $H$ and $H'$ are monoidally Morita equivalent if and only if they are {\em gauge equivalent}, {\em i.e.}, $H'$ is obtained from $H$ by twisting via a gauge transformation in the sense of Drinfeld. However, in general, it is not easy to check this condition directly. Thus it is important to study {\em gauge invariants}, {\em i.e.}, properties or quantities invariant under the gauge equivalence.

Nowadays several gauge invariants have been introduced and studied \cite{MR1689203,MR1909645,KMN09,MR2104908,MR2313527,MR2381536,MR2366965,MR2725181,MR2564847,KenichiShimizu:2012,MR2512416,MR2580652}. The main subject of this paper is the gauge invariant $\nu_n$, called the {\em $n$-th indicator}, introduced by Kashina, Montgomery and Ng in \cite{KMN09}. Motivated by the formula of the $n$-th Frobenius-Schur indicator of the regular representation of a semisimple Hopf algebra \cite{MR2213320,MR1808131}, for a finite-dimensional Hopf algebra $H$ over $k$ and an integer $n \ge 1$, they defined $\nu_n(H) \in k$ to be the trace of
\begin{equation}
  \label{eq:S-Sw-power}
  h \mapsto S_H(h^{[n-1]}) \quad (h \in H),
\end{equation}
where $S_H$ is the antipode of $H$ and $h^{[m]}$ is the $m$-th Sweedler power of $h \in H$ given by $h^{[m]} = h_{(1)} \dotsb h_{(m)}$ in the Sweedler notation. In particular, the second indicator is equal to the trace of the antipode.

If $H = k G$ is the group algebra of a finite group $G$, then~\eqref{eq:S-Sw-power} reduces to the linear map $k G \to k G$ given by $g \mapsto g^{-n + 1}$ ($g \in G$). Hence, we have
\begin{equation}
  \label{eq:ind-group-algebra}
  \nu_n(k G) = \# \{ g \in G \mid g^n = 1 \}
\end{equation}
for every $n \ge 1$. Since the right-hand side of~\eqref{eq:ind-group-algebra} is a very basic subject in finite group theory, we expect that the indicators could be an interesting subject in the theory of Hopf algebras. The aim of this paper is to study basic properties of the indicators of finite-dimensional Hopf algebras.

The present paper is organized as follows: In Section~\ref{sec:preliminaries}, we collect some basic results on Hopf algebras. In Section~\ref{sec:ind-Hopf-alg}, we study basic properties of the indicators and provide some methods to compute them. For simplicity, suppose $k = \mathbb{C}$. Then our results in Section~\ref{sec:ind-Hopf-alg} imply
\begin{equation*}
  \nu_n(H) \in \mathbb{Z}[e^{2\pi\sqrt{-1}/N}]
  \text{\quad (cyclotomic integrality)},
\end{equation*}
where $N = n \cdot \text{(the order of $S_H^2$)}$, and
\begin{equation*}
  \nu_n(H^{\op}) = \nu_n(H^{\cop}) = \overline{\nu_n(H)}
\end{equation*}
for all $n \ge 1$. From this, we obtain the formula $\nu_n(D(H)) = |\nu_n(H)|^2$ (Corollary~\ref{cor:ind-Drinfeld-double}), which has been conjectured in \cite{KMN09}. We also prove that if $H$ is filtered, then $\nu_n(H) = \nu_n(\gr H)$ (Theorem~\ref{thm:ind-filtered}).

In \S\ref{subsec:ind-non-positive}, we extend $\nu_n$ for any integer $n \in \mathbb{Z}$. Since the $m$-th Sweedler power map is defined for all $m \in \mathbb{Z}$, we can define $\nu_n(H)$ to be the trace of \eqref{eq:S-Sw-power} for all $n \in \mathbb{Z}$. We will prove that $\nu_n$ is a gauge invariant not only for all $n \ge 1$ but also for all $n \le 0$ (Theorem \ref{thm:ind-gauge-inv}). However,
\begin{equation*}
  \nu_{-n}(H) = \nu_{-1}(H) \cdot \overline{\nu_n(H)}
\end{equation*}
holds for all $n > 0$. Thus the values of $\nu_0$ and $\nu_{-1}$ are especially interesting; see Propositions~\ref{prop:ind-0} and~\ref{prop:ind-(-1)}.

Note that (\ref{eq:ind-group-algebra}) is periodic in $n$. In \cite{KMN09}, they showed that $\nu_H := \{ \nu_n(H) \}_{n \ge 1}$ is not periodic in general but is linearly recursive. Hence we can consider the minimal polynomial $\phi_H(X)$ of the sequence $\nu_H$. In Section \ref{sec:min-pol-ind}, we prove that every root of $\phi_H(X)$ is a root of unity (Theorem~\ref{thm:min-pol-roots}). This implies that, for a finite-dimensional Hopf algebra $H$, there are finitely many periodic sequences $c_0, \dotsc, c_{r}$ such that
\begin{equation*}
  \nu_n(H) = c_0(n) + \binom{n}{1} c_1(n) + \dotsb + \binom{n}{r} c_{r}(n)
\end{equation*}
for all $n \ge 1$ (Theorem~\ref{thm:ind-behavior}). We also discuss the relation between $\phi_H(X)$ and the quasi-exponent of $H$ in some restricted cases.

In Section~\ref{sec:appl-pointed}, we apply our results to the finite-dimensional pointed Hopf algebra $u(\mathcal{D}, \lambda, \mu)$ introduced by Andruskiewitsch and Schneider \cite{MR2630042}. Important examples of $u(\mathcal{D}, \lambda, \mu)$ are the Taft algebra $T_{N^2}(\omega)$ and the small quantum group $u_q(\mathfrak{g})$. We apply our results to $u_q(\mathfrak{sl}_2)$ and show that if $p$ and $q$ are roots of unity of the same order, then $u_p(\mathfrak{sl}_2)$ and $u_q(\mathfrak{sl}_2)$ are gauge equivalent if and only if $p = q$.

In Section~\ref{sec:higher-ind-taft}, we express the $n$-th indicator of $T_{N^2}(\omega)$ by using the generating function of certain type of partitions studied in \cite{MR1462505} (Theorem~\ref{thm:higher-ind-Taft}). We also show that if $q$ is of odd order $N > 1$, then
\begin{equation*}
  \nu_n(u_q(\mathfrak{sl}_2)) = \frac{1}{\gcd(N, n)} | \nu_n(T_{N^2}(q^2)) |^2
\end{equation*}
for all $n \ge 1$ (Theorem \ref{thm:higher-ind-uqsl2}). By using this result, we show $\nu_n(u_q(\mathfrak{sl}_2)) = n^2$ for all $n \ge 1$ if $q$ is a primitive third root of unity.

\subsection*{Acknowledgments}

The author is supported by Grant-in-Aid for JSPS Fellows (24$\cdot$3606).

\section{Preliminaries}
\label{sec:preliminaries}

\subsection{Notations and conventions}
\label{subsec:notations}

Throughout this paper, we work over a fixed field $k$. Given two vector spaces $V$ and $W$ (over $k$), we denote by $V \otimes W$ their tensor product over $k$. The dual space of $V$ is denoted by $V^*$. We always identify $k \otimes V = V = V \otimes k$ and regard $V^* \otimes W^*$ as a subset of $(V \otimes W)^*$.

By a {\em grading} or a {\em filtration}, we usually mean those indexed by the set $\mathbb{Z}_{\ge 0}$ of non-negative integers. Thus, a {\em graded vector space} is a vector space $V$ endowed with a decomposition $V = \bigoplus_{n = 0}^\infty V(n)$, and a {\em filtered vector space} is a vector space $V$ endowed with a sequence $V_0 \subset V_1 \subset \dotsb \subset V$ of subspaces such that $V = \bigcup_{n = 0}^\infty V_n$.

We refer the reader to \cite{MR1786197,MR1243637,MR0252485} for the basic theory of Hopf algebras and to \cite{MR1321145,MR1712872} for the basic theory of monoidal categories. Given a Hopf algebra $H$, we denote the $k$-linear monoidal category of left $H$-modules and its full subcategory of finite-dimensional objects by $\Mod{H}$ and $\fdMod{H}$, respectively.  The multiplication, the comultiplication, the counit and the antipode of $H$ are denoted by $\nabla_H$, $\Delta_H$, $\varepsilon_H$, and $S_H$, respectively, or simply by $\nabla$, $\Delta$, $\varepsilon$ and $S$ if there are no confusion. For each $n \ge 1$, we define 
\begin{equation}
  \label{eq:iterated-mult}
  \nabla^{(n)} = \nabla^{(n)}_H: H^{\otimes n} \to H, \quad \nabla^{(n)}(h_1 \otimes \dotsb \otimes h_n) = h_1 \dotsb h_n \quad (h_i \in H).
\end{equation}
Dually, we define $\Delta^{(n)}_{\mathstrut} = \Delta^{(n)}_H: H \to H^{\otimes n}$ inductively by
\begin{equation}
  \label{eq:iterated-comul}
  \Delta^{(1)} = \id_H \text{\quad and \quad}
  \Delta^{(n+1)} = (\Delta^{(n)} \otimes \id_H) \circ \Delta.
\end{equation}
We use the Sweedler notation: $\Delta^{(n)}(h) = h_{(1)} \otimes \dotsb \otimes h_{(n)}$ ($h \in H$).

\subsection{Gauge invariants}
\label{subsec:gauge-inv}

Let $H$ be a Hopf algebra. An invertible element $F \in H \otimes H$ is called a {\em gauge transformation} if $(\varepsilon \otimes \id)(F) = 1 = (\id \otimes \varepsilon)(F)$. Given such an element $F \in H \otimes H$, we define an algebra map
\begin{equation*}
  \Delta_F: H \to H \otimes H,
  \quad \Delta_F(h) = F \Delta(h) F^{-1}
  \quad (h \in H).
\end{equation*}
The triple $H_F = (H, \Delta_F, \varepsilon)$ is not even a bialgebra in general, but can be a quasi-Hopf algebra by defining additional structures by using $F$. We denote by $X_F$ the left $H$-module $X$ regarded as a left $H_F$-module. Multiplying $F$ induces a canonical isomorphism $(X \otimes Y)_F \to X_F \otimes Y_F$. As Drinfeld showed in \cite{MR1047964}, the identity functor on $\Mod{H}$ extends to an isomorphism
\begin{equation}
  \label{eq:gauge-eq-1}
  (-)_F: \Mod{H} \to \Mod{H_F}, \quad X \mapsto X_F
\end{equation}
of $k$-linear monoidal categories.

If $F$ satisfies $(F \otimes 1) \cdot (\Delta \otimes \id_H)(F) = (1 \otimes F) \cdot (\id_H \otimes \Delta)(F)$, then $H_F$ is indeed a Hopf algebra with antipode
\begin{equation}
  \label{eq:gauge-eq-2}
  S_{H_F}(h) = \beta_F S_H(h) \beta_F^{-1} \quad (h \in H),
\end{equation}
where $\beta_F = \sum_i a_i S(b_i)$, $F = \sum_i a_i \otimes b_i$; see, {\em e.g.}, \cite{MR1321145} for details. Such an element $F$ is often called a {\em dual 2-cocycle} or a {\em twist} of $H$.

Now let $H$ and $H'$ be Hopf algebras. They are said to be {\em gauge equivalent} if there exists a twist $F$ of $H$ such that $H' \cong H_F$. By the result of Schauenburg \cite{MR1408508}, if both $H$ and $H'$ are finite-dimensional, then the following are equivalent:
\begin{enumerate}
\item[(1)] $\Mod{H} \approx \Mod{H'}$ as $k$-linear monoidal categories;
\item[(2)] $H$ and $H'$ are gauge equivalent.
\end{enumerate}
Given a class $\mathfrak{C}$ of Hopf algebras, a {\em gauge invariant} \cite{KMN09} for $\mathfrak{C}$ is a quantity $\nu(H)$ defined for each $H \in \mathfrak{C}$ such that $\nu(H) = \nu(H')$ whenever $H, H' \in \mathfrak{C}$ are gauge equivalent. In this paper, we mainly consider gauge invariants for the class $\mathfrak{C}_{\mathit{fd}}$ of finite-dimensional Hopf algebras. Thus, by a gauge invariant, we always mean that for $\mathfrak{C}_{\mathit{fd}}$.

\subsubsection*{Two-cocycle deformation}

Let $C$ be a coalgebra, and let $A$ be an algebra. Recall that $\Hom_k(C, A)$ is an algebra with respect to the convolution product $\star$ defined by $(f \star g)(c) = f(c_{(1)}) g(c_{(2)})$ ($f, g: C \to A$, $c \in C$). A {\em 2-cocycle} of a Hopf algebra $H$ is a linear map $\sigma: H \otimes H \to k$ being invertible with respect to the convolution product and satisfying $\sigma(x \otimes 1) = \varepsilon(x) = \sigma(1 \otimes x)$ and
\begin{gather*}
  \sigma(x_{(1)} \otimes y_{(1)}) \cdot \sigma(x_{(2)} y_{(2)} \otimes z)
  = \sigma(y_{(1)} \otimes z_{(1)}) \cdot \sigma(x \otimes y_{(2)} z_{(2)})
\end{gather*}
for all $x, y, z \in H$. If $\sigma$ is a 2-cocycle of $H$, then the coalgebra $H$ is a bialgebra with multiplication $*$ defined by $x * y = \sigma(x_{(1)},y_{(1)}) x_{(2)} y_{(2)} \tilde{\sigma}(x_{(3)},y_{(3)})$ for $x, y \in H$, where $\tilde{\sigma}$ is the inverse of $\sigma$ with respect to $\star$. We denote this new bialgebra by $H^\sigma$ and call it the {\em 2-cocycle deformation} of $H$. $H^\sigma$ is in fact a Hopf algebra (see \cite{MR1213985} for the description of its antipode).

Let $\sigma$ be a 2-cocycle of a finite-dimensional Hopf algebra $H$. Then we can regard $\sigma$ as a Drinfeld twist of the dual Hopf algebra $H^*$. If we do so, then $(H^*)_\sigma$ is isomorphic to $(H^\sigma)^*$ as a Hopf algebra. Hence, if $\nu$ is a gauge invariant such that $\nu(X) = \nu(X^*)$ for all $X \in \mathfrak{C}_{\mathit{fd}}$, then we have
\begin{equation*}
  \nu(H^\sigma) = \nu((H^\sigma)^*) = \nu((H^*)_\sigma) = \nu(H^*) = \nu(H)
\end{equation*}
for all $H \in \mathfrak{C}_{\mathit{fd}}$. Summarizing, we conclude:

\begin{lemma}
  \label{lem:inv-cocycle-deform}
  If a gauge invariant is invariant under taking the dual, then it is invariant under 2-cocycle deformation.
\end{lemma}

\subsection{Exponent of Hopf algebras}
\label{subsec:exponent}

The exponent and the quasi-exponent are gauge invariants introduced by Etingof and Gelaki \cite{MR1689203,MR1909645}. Let $H$ be a finite-dimensional Hopf algebra, and let $u$ be the Drinfeld element of the Drinfeld double $D(H)$. The {\em exponent} of $H$ is defined and denoted by
\begin{equation*}
  \exp(H) = \ord(u) \text{\quad (the order of $u$)}.
\end{equation*}
By the results of \cite{MR1689203,MR1909645}, $u^n$ is unipotent for some $n > 0$. Thus the set
\begin{equation*}
  \{ n \in \mathbb{Z} \mid \text{$u^n$ is unipotent} \}
\end{equation*}
is a non-zero ideal of $\mathbb{Z}$. The {\em quasi-exponent} of $H$, denoted by $\qexp(H)$, is defined to be the smallest positive element of this ideal. Note that $\exp(H)$ is possibly infinite while $\qexp(H)$ is always finite.

The exponent and the quasi-exponent can be defined without using the Drinfeld element. For $n \ge 0$, we define $T_H^{(n)}: H \to H$ by $T_H^{(0)}(h) = \varepsilon(h)1_H$ and
\begin{equation}
  \label{eq:modified-Sw-pow}
  T_H^{(n)}(h) = h_{(1)} S^{-2}(h_{(2)}) \dotsb S^{-2n+2}(h_{(n)})
  \quad (n \ge 1)
\end{equation}
for $h \in H$. Then $\exp(H) = \min \{ n > 0 \mid T_H^{(n)} = T_H^{(0)} \}$ (with the convention $\min \emptyset = \infty$) and $\qexp(H)$ is the smallest integer $n > 0$ such that
\begin{equation*}
  \sum_{j = 0}^m (-1)^j \binom{m}{j} T_H^{(n j)} = 0
\end{equation*}
for some integer $m > 0$. These characterizations follow from the fact that for all polynomial $f(X) = \sum_{i = 0}^m c_i X^i \in k[X]$,
\begin{equation}
  \label{eq:qexp-1}
  f(u) = 0 \iff \sum_{i = 0}^m c_i T_H^{(i)} = 0.
\end{equation}
Fact \eqref{eq:qexp-1} is a part of \cite[Proposition 2.4]{MR1909645}; strictly speaking, \eqref{eq:qexp-1} is proved in \cite{MR1909645} under the assumption $k = \mathbb{C}$ but one can check that the same proof is valid in the case of general $k$ ({\em cf}. \cite[Remark 4.15]{MR1909645}).

\subsubsection*{Exponent in characteristic zero}

Now suppose $\mathrm{char}(k) = 0$. In this case, the gauge invariance of the exponent and the quasi-exponent is 
showed in \cite{MR1689203,MR1909645}. Following them, whenever $\exp(H) < \infty$,
\begin{equation}
  \label{eq:qexp-char-0}
  \exp(H) = \qexp(H)
\end{equation}
and if $H$ is semisimple, then $\exp(H)$ is finite and divides $\dim_k(H)^3$. Kashina \cite{MR1669152,MR1763611} conjectured that $\exp(H)$ divides $\dim_k(H)$ if $H$ is semisimple. This conjecture is still open, but a lot of interesting results on the exponent of semisimple Hopf algebras have been obtained \cite{MR1906068,MR2213320,MR2365863}.

\subsubsection*{Exponent in positive characteristic}

Suppose $p := \mathrm{char}(k) > 0$. Then $\exp(H)$ is always finite \cite{MR1689203}. The gauge invariance of the exponent is showed also in \cite{MR1689203}. Unlike the case of $\mathrm{char}(k) = 0$, \eqref{eq:qexp-char-0} does not hold in general; instead,
\begin{equation}
  \label{eq:qexp-char-p}
  \qexp(H) = \exp(H)_p'
\end{equation}
holds in general, where, for a positive integer $n$, $n_p'$ means the $p$-prime part of $n$, {\it i.e.}, the largest divisor of $n$ relatively prime to $p$. This is the special case of the following lemma where $A = D(H)$ and $x$ is the Drinfeld element.

\begin{lemma}
  \label{lem:unipo-ord}
  Let $A$ be an algebra over a field of characteristic $p > 0$. If $x \in A$ is an element of finite order, then we have
  \begin{equation*}
    \ord(x)_p' = \min \{ n > 0 \mid \text{\rm $x^n$ is unipotent} \}.
  \end{equation*}
\end{lemma}
\begin{proof}
  Let $a$ and $b$ be the left and the right-hand side of the above equation, respectively. We first show $b \mid a$. Since $x^{\ord(x)} = 1$ is unipotent, $b \mid \ord(x)$. If $b$ was not relatively prime to $p$, then $c = b/p$ would be a positive integer and
  \begin{equation*}
    (x^c - 1)^{m p} = (x^{c p} - 1)^m = (x^b - 1)^m = 0.
  \end{equation*}
  This contradicts the minimality of $b$ and thus $b$ is relatively prime to $p$. Therefore $b$ divides the $p$-prime part of $\ord(x)$, {\em i.e.}, $a = \ord(x)_p'$.

  Next we show $a \mid b$. By the definition of $b$, there exists an integer $m > 0$ such that $(x^b - 1)^m = 0$. Choose an integer $\mu$ so that $p^\mu > m$. Then we have
  \begin{equation*}
    x^{b p^\mu} - 1 = x^{b p^\mu} - 1^{p^\mu} = (x^b - 1)^{p^\mu} = 0.
  \end{equation*}
  Thus $\ord(x) \mid b p^\mu$. Since $a \mid \ord(x)$, and since $a$ is relatively prime to $p$ by definition, we see that $a$ divides the $p$-prime part of $b p^\mu$, {\em i.e.}, $b$.
\end{proof}

The gauge invariance of the quasi-exponent in positive characteristic has not been discussed in detail ({\em cf}. \cite[Remark 4.15]{MR1909645}). In view of \eqref{eq:qexp-char-p}, we could say that the gauge invariance of the quasi-exponent in positive characteristic follows from the gauge invariance of the exponent.

Let $H$ be a finite-dimensional Hopf algebra. If $g \in H$ is a grouplike element, then we have $\ord(g) \mid \exp(H)$ \cite[Proposition~2.2]{MR1689203}. Hence, by~\eqref{eq:qexp-char-p},
\begin{equation}
  \label{eq:qexp-grplike-ord}
  \ord(g)_{\mathrm{char}(k)}' \mid \qexp(H).
\end{equation}
Note that \eqref{eq:qexp-grplike-ord} holds even in the case where $\mathrm{char}(k) = 0$ \cite[Proposition~2.6]{MR1909645} if we would use the convention $n_0' = n$ for a positive integer $n$.

\subsection{Coradical filtration}

Given a coalgebra $C$, we denote by $\corad(C)$ the coradical of $C$, {\em i.e.}, the sum of all simple subcoalgebras of $C$. For each $n \ge 0$, we define $C_n \subset C$ inductively by
\begin{equation*}
  C_0 = \corad(C) \text{\quad and \quad} C_n = C_0 \wedge C_{n-1},
\end{equation*}
where $X \wedge Y = \Delta^{-1}(X \otimes C + C \otimes Y)$ for $X, Y \subset C$. The sequence $C_0 \subset C_1 \subset \dotsb$ is called the {\em coradical filtration} of $C$. It is known that the coradical filtration is a coalgebra filtration: $\Delta(C_n) \subset \sum_{i = 0}^n C_i \otimes C_{n-i}$. For a right $C$-comodule $M$ with coaction $\rho_M: M \to M \otimes C$, we define the {\em Loewy length} of $M$ by
\begin{equation*}
  \Lw(M) = \min \{ n \ge 0 \mid \rho_M(M) \subset M \otimes C_{n - 1} \}
\end{equation*}
with convention $C_{-1} = 0$ and $\min \emptyset = \infty$ ({\em cf}. \cite[Lemma~2.2]{MR2506832}). If we regard $C$ as a right $C$-comodule by the comultiplication, then $\Lw(C) = \min \{ n \ge 0 \mid C_{n - 1} = C \}$. Since $C = \bigcup_{n = 0}^\infty C_n$, we have
\begin{equation}
  \label{eq:Loewy-length}
  \Lw(C) = \min \{ n \ge 0 \mid C_{n - 1} = C_n \}.
\end{equation}
The following lemma is well-known and proved by induction on $n$.

\begin{lemma}
  \label{lem:nilpo}
  Let $C$ be a coalgebra with coradical filtration $\{ C_n \}_{n \ge 0}$, and let $A$ be an algebra. If $f: C \to A$ is a linear map such that $f|_{C_0} = 0$, then
  \begin{equation*}
    f^{\star n}|_{C_{n - 1}} = 0
  \end{equation*}
  for all $n \ge 1$, where $f^{\star n}$ is the $n$-th power of $f$ with respect to $\star$. In particular, if $\ell = \Lw(C)$ is finite, then $f^{\star \ell} = 0$.
\end{lemma}

\subsection{Braided Hopf algebras}
\label{subsec:braided-Hopf}

Given a Hopf algebra $H$ with bijective antipode, we denote by ${}^H_H \YD$ the category of left Yetter-Drinfeld modules over $H$ \cite[Chapter 10]{MR1243637}. Namely, an object of ${}^H_H \YD$ is a left $H$-module $V$, which is a left $H$-comodule at the same time, such that the Yetter-Drinfeld condition
\begin{equation}
  \label{eq:YD-condition}
  \rho_V(h \rightharpoonup v) = h_{(1)} v_{(-1)} S(h_{(3)}) \otimes (h_{(2)} \rightharpoonup v_{(0)})
\end{equation}
holds for all $h \in H$ and $v \in V$. Here, $\rightharpoonup$ is the action and $\rho_V(v) = v_{(-1)} \otimes v_{(0)}$ is the coaction of $H$. Note that ${}^H_H \YD$ has a braiding given by
\begin{equation*}
  c_{V,W}(v \otimes w) = (v_{(-1)} \rightharpoonup w) \otimes v_{(0)}
  \quad (V, W \in {}^H_H \YD; v \in V, w \in W).
\end{equation*}
Let $\Gamma$ be an abelian group. The case where $H = k \Gamma$ is important for later use. If this is the case, then we write ${}^{k \Gamma}_{k \Gamma} \YD$ as ${}^\Gamma_\Gamma \YD$ for simplicity. By~\eqref{eq:YD-condition}, an object of ${}^\Gamma_\Gamma \YD$ is nothing but a $\Gamma$-graded vector space $V = \bigoplus_{g \in \Gamma} V_g$ such that $x \rightharpoonup V_g = V_{g}$ for all $x, g \in \Gamma$.

Now we go back to the general situation. If $A$ and $A'$ are algebras in ${}^H_H \YD$, then $A \otimes A'$ is naturally an algebra in ${}^H_H \YD$ with multiplication given by
\begin{equation*}
  \begin{CD}
    (A \otimes A') \otimes (A \otimes A')
    @>{\id \otimes c_{A,A'} \otimes \id}>>
    A \otimes A \otimes A' \otimes A'
    @>{\nabla \otimes \nabla'}>>
    A \otimes A',
  \end{CD}
\end{equation*}
where $\nabla$ and $\nabla'$ are the multiplications of $A$ and $A'$, respectively. This algebra is denoted by $A \mathop{\underline{\otimes}} A'$ and called the {\em braided tensor product} of $A$ and $A'$. In a similar way, the braided tensor product of coalgebras in ${}^H_H \YD$ is defined.

A Hopf algebra in ${}^H_H \YD$ is called a {\em braided Hopf algebra over $H$}. By definition, it is an ordinary algebra, an ordinary coalgebra, and a left Yetter-Drinfeld module over $H$. Now let $B$ be a braided Hopf algebra over $H$ with multiplication $\nabla$ and comultiplication $\Delta$. To avoid confusion with the coaction of $H$, we denote the comultiplication as $\Delta(b) = b^{<1>} \otimes b^{<2>}$ ($b \in B$). It is known that the antipode $S$ of $B$ is an anti-algebra map and an anti-coalgebra map in the `braided sense'. Namely,
\begin{equation}
  \label{eq:br-Hopf-antipode}
  S \circ \nabla = \nabla \circ c_{B,B} \circ (S \otimes S)
  \text{\quad and \quad}
  \Delta \circ S = (S \otimes S) \circ c_{B,B} \circ \Delta.
\end{equation}
For each $a, b \in B$, their {\em braided commutator} $[a,b]_c$ is defined by
\begin{equation}
  \label{eq:br-commutator}
  [a,b]_c := \nabla (\id_{B \otimes B} - c_{B,B}^{})(a \otimes b).
\end{equation}

Now let $B$ and $B'$ be braided Hopf algebras over $H$. Then $B \mathop{\underline{\otimes}} B'$ is both an algebra and a coalgebra in ${}^H_H \YD$. However, in general, one cannot show that $B \mathop{\underline{\otimes}} B'$ is a bialgebra in ${}^H_H \YD$. So we assume that $B$ and $B'$ {\em centralize each other} \cite{MR1990929}, {\it i.e.}, $c_{B',B} \circ c_{B,B'} = \id_{B \otimes B'}$. Then $B \mathop{\underline{\otimes}} B'$ is shown to be a bialgebra in ${}^H_H \YD$. This is in fact a braided Hopf algebra over $H$ with antipode $S_{B \mathop{\underline{\otimes}} B'} = S_B \otimes S_{B'}$.

\subsubsection*{Braided tensor algebra} Let $V \in {}^H_H \YD$. The ordinary tensor algebra $T(V)$ over $V$ is naturally an algebra in ${}^H_H \YD$. The linear map
\begin{equation*}
  V \to T(V) \mathop{\underline{\otimes}} T(V), \quad v \mapsto v \otimes 1 + 1 \otimes v \quad (v \in V)
\end{equation*}
can be extended to an algebra map $\Delta: T(V) \to T(V) \mathop{\underline{\otimes}} T(V)$. $T(V)$ becomes a braided Hopf algebra over $H$ with this $\Delta$ and is called the {\em braided tensor algebra} over $V$.

\subsubsection*{Nichols algebra} Let $V \in {}^H_H \YD$. A {\em Nichols algebra} of $V$ is a graded braided Hopf algebra $B = \bigoplus_{n \ge 0} B(n)$ over $H$ satisfying the following conditions:
\begin{enumerate}
\item[(1)] $B(0)$ is the trivial Yetter-Drinfeld module.
\item[(2)] $B(1)$ is isomorphic to $V$ and generates $B$ as an algebra.
\item[(3)] $B(1) = \{ b \in B \mid \Delta(b) = b \otimes 1 + 1 \otimes b \}$.
\end{enumerate}
A Nichols algebra of $V$ always exists and is unique up to isomorphisms. Hence we write it as $\mathfrak{B}(V) = \bigoplus_{n = 0}^\infty \mathfrak{B}_n(V)$. Some constructions of the Nichols algebra are known but omitted here for brevity. For more details, see, {\it e.g.}, \cite[\S2]{MR1913436}.

\subsection{Bosonization}
\label{subsec:bosonization}

Let $H$ be a Hopf algebra with bijective antipode. Given a braided Hopf algebra $B$ over $H$, we can construct an ordinary Hopf algebra $B \# H$, called the {\em bosonization} \cite{MR1257312,MR778452}. As a vector space, $B \# H = B \otimes H$. The multiplication and the comultiplication are given respectively by
\begin{gather*}
  (b \# h) \cdot (b' \# h') = b (h_{(1)} \rightharpoonup b') \# h_{(2)} h', \\
  \Delta(b \# h) = b^{<1>} \# (b^{<2>})_{(-1)} h_{(1)} \otimes (b^{<2>})_{(0)} \# h_{(2)}
\end{gather*}
for $b, b' \in B$ and $h, h' \in H$. Here, for $b \in B$ and $h \in H$, we write $b \otimes h \in B \otimes H$ as $b \# h$. Under the same notation, the antipode of $B \# H$ is given by
\begin{equation}
  \label{eq:boson-antipode}
  S_{B \# H}(b \# h) = (1 \# S_H(b_{(-1)} h)) \cdot (S_B(b_{(0)}) \# 1)
  \quad (b \in B, h \in H).
\end{equation}

A Hopf algebra with projection onto $H$ is always of the form $B \# H$; namely, let $A$ be a Hopf algebra having $H$ as a Hopf subalgebra, and suppose that there exists a Hopf algebra map $\pi: A \to H$ such that $\pi|_H = \id_H$. Then
\begin{equation*}
  B := \{ a \in A \mid (\id_A \otimes \pi)(a) = a \otimes 1 \}
\end{equation*}
has a structure of a braided Hopf algebra over $H$ such that the map
\begin{equation}
  \label{eq:boson-Radford-iso}
  B \# H \to A,
  \quad b \# h \mapsto b \cdot h
  \quad (b \in B, h \in H)
\end{equation}
is an isomorphism of Hopf algebras; see \cite{MR778452} for details.

Now we consider the case where $A = \bigoplus_{n \ge 0} A(n)$ is a graded Hopf algebra such that $A(0) = H$ and $\pi: A \to H$ is the homogeneous projection. Then $B$ is a graded braided Hopf algebra over $H$ with grading $B(n) = B \cap A(n)$. Moreover, identifying $B \# H$ with $A$ via~\eqref{eq:boson-Radford-iso}, we have
\begin{equation}
  \label{eq:boson-free-over-A0}
  B(n) \# H = A(n)
\end{equation}
for all $n \ge 0$ \cite[\S2]{MR1659895}.

Recall that a Hopf algebra $\mathcal{A}$ is said to have the {\em dual Chevalley property} if its coradical is a Hopf subalgebra. If this is the case, then $\mathcal{A}$ is a filtered Hopf algebra with respect to the coradical filtration. Hence we can consider the associated graded Hopf algebra $A := \gr \mathcal{A}$, $A(n) = \mathcal{A}_n / \mathcal{A}_{n - 1}$ (with $\mathcal{A}_{-1} = 0$).

Suppose moreover $\dim_k \mathcal{A} < \infty$. Then the Loewy length $\ell := \Lw(\mathcal{A})$ of the coalgebra $\mathcal{A}$ is finite. By~\eqref{eq:Loewy-length} and \eqref{eq:boson-free-over-A0}, we have an inequality
\begin{equation*}
  \dim_k(\mathcal{A})
  = \sum_{n = 0}^{\ell-1} \dim_k(A(n))
  = \dim_k(H) \sum_{n = 0}^{\ell-1} \dim_k(B(n))
  \ge \ell \cdot \dim_k(H).
\end{equation*}
Summarizing the above arguments (and rewriting $\mathcal{A}$ as $A$), we obtain:

\begin{lemma}
  \label{lem:ll-bound}
  If a finite-dimensional Hopf algebra $A$ has the dual Chevalley property, then $\Lw(A) \cdot \dim_k(\corad(A)) \le \dim_k(A)$.
\end{lemma}

Let $B$ be a braided Hopf algebra over $H$. By the result of Radford \cite{MR778452} mentioned above, $(B \# H)^{\op}$ is a Hopf algebra of the form $B' \# H^{\op}$ for some braided Hopf algebra $B'$ over $H^{\op}$. An explicit description of $B'$ is found in \cite{RS05}, where our $B'$ is denoted by $B^{\underline{op}}$.

We are interested in the case where $H = k \Gamma$ for some abelian group $\Gamma$. If this is the case, then $H^{\op} = H$ and therefore the above $B'$ is again a braided Hopf algebra over $H$. Given $X \in {}^\Gamma_\Gamma \YD$ with action $\rightharpoonup$, we define $X^{\op} \in {}^\Gamma_\Gamma \YD$ to be the vector space $X$ equipped with the same coaction as $X$ and the new action $\rightharpoondown$ given by $g \rightharpoondown x = g^{-1} \rightharpoonup x$ ($g \in \Gamma$, $x \in X$). By \cite{RS05}, we obtain:

\begin{lemma}
  \label{lem:boson-op}
  Let $\Gamma$ be an abelian group, and let $V \in {}^\Gamma_\Gamma \YD$. Then there is an isomorphism of graded Hopf algebras $(\mathfrak{B}(V) \# k \Gamma)^{\op} \cong \mathfrak{B}(V^{\op}) \# k \Gamma$.
\end{lemma}

\section{Indicators of Hopf algebras}
\label{sec:ind-Hopf-alg}

\subsection{Definition of the indicators}
\label{subsec:ind-Hopf-definition}

Let $H$ be a Hopf algebra. For an integer $m$, the {\em $m$-th Sweedler power map}
\begin{equation*}
  P_H^{(m)}: H \to H, \quad h \mapsto h^{[m]} \quad (h \in H)
\end{equation*}
is defined to be the $m$-th power of the identity map $\id_H: H \to H$ with respect to the convolution product. Thus, for all $h \in H$ and an integer $m \ge 1$,
\begin{equation*}
  h^{[0]} = \varepsilon(h)1_H,
  \quad h^{[m]} = h_{(1)} \dotsb h_{(m)}
  \quad \text{and} \quad
  h^{[-m]} = S(h_{(1)}) \dotsm S(h_{(m)}).
\end{equation*}
Now we suppose that $H$ is finite-dimensional. Following \cite{KMN09}, for each integer $n \ge 1$, we define the {\em $n$-th indicator} of $H$ by
\begin{equation}
  \label{eq:ind-def}
  \nu_n(H) = \Trace \Big( S_H^{} \circ P_H^{(n-1)}: H \to H \Big),
\end{equation}
where $\Trace$ means the ordinary trace of a linear operator. In \cite{KMN09}, they showed that the $n$-th indicator $\nu_n$ is a gauge invariant for all $n \ge 1$. They also showed that $\nu_n$ can be expressed by using integrals:

\begin{lemma}
  \label{lem:ind-integral}
  If $\lambda \in H^*$ and $\Lambda \in H$ are both left integrals or both right integrals such that $\langle \lambda, \Lambda \rangle = 1$, then we have $\nu_n(H) = \langle \lambda, \Lambda^{[n]} \rangle$ for all $n \ge 1$.
\end{lemma}

This is Corollary~2.6 of \cite{KMN09}. We include the proof for later use.

\begin{proof}
  Consider the case where both $\lambda$ and $\Lambda$ are left integrals. Then, by \cite{MR1265853}, $\lambda' = \lambda \circ S^{-1}$ is a right integral such that $\langle \lambda', \Lambda \rangle = 1$. Hence, by the Radford trace formula \cite{MR1265853}, we have
  \begin{equation*}
    \Trace(F)
    = \langle \lambda', S(\Lambda_{(2)}) F(\Lambda_{(1)}) \rangle
    = \langle \lambda, S^{-1}F(\Lambda_{(1)}) \Lambda_{(2)} \rangle
    = \langle \lambda, (S^{-1}F \star \id_H)(\Lambda) \rangle
  \end{equation*}
  for all $F \in \End_k(H)$. Applying this formula to $F = S_H^{} \circ P_H^{(n-1)}$, we obtain the desired formula. The proof is almost the same in the case where $\lambda$ and $\Lambda$ are both right integrals.
\end{proof}

Set $I_n(H) = \Hom_H(H^{\otimes n}, k)$ for $n \ge 1$. For simplicity of notation, a linear map $H^{\otimes n} \to k$ is often regarded as a multilinear map $H \times \dotsb \times H \to k$. Now we prove:

\begin{lemma}
  \label{lem:ind-rot-map}
  $\nu_n(H) = \Trace(\mathcal{E}_H^{(n)})$ for all $n \ge 1$, where $\mathcal{E}_H^{(n)}: I_n(H) \to I_n(H)$ is a linear map given by
  \begin{equation*}
    \mathcal{E}_H^{(n)}(f)(a_1, \dotsc, a_n) = f(a_2, \dotsc, a_n, S^2(a_1))
    \quad (f \in I_n(H), a_1, \dotsc, a_n \in H).
  \end{equation*}
\end{lemma}
\begin{proof}
  Recall that for each $X \in \Mod{H}$, there is a canonical isomorphism
  \begin{equation*}
    \Psi_X: \Hom_H(H \otimes X, k) \to X^*,
    \quad f \mapsto f(1_H \otimes -).
  \end{equation*}
  Put $X = H^{\otimes (n - 1)}$ and define $\mathcal{E}': X \to X$ by
  \begin{equation*}
    \mathcal{E}'(x)
    = S(h_{(n-1)}) a_2 \otimes \dotsb \otimes S(h_{(2)}) a_{n - 1} \otimes S(h_{(1)})
  \end{equation*}
  for $x = h \otimes a_2 \otimes \dotsb \otimes a_{n - 1} \in X$. For all such $x$ and $f \in I_n(H)$, we compute
  \begin{align*}
    \Psi_X(f)(\mathcal{E}'(x))
    & = f(1, S(h_{(n-1)}) a_2, \dotsc, S(h_{(2)}) a_{n - 1}, S(h_{(1)})) \\
    & = f(S(h_{(n)}) h_{(n+1)}, S(h_{(n-1)}) a_2, \dotsc, S(h_{(2)}) a_{n - 1}, S(h_{(1)})) \\
    & = f(h, a_2, \dotsc, a_{n - 1}, 1)
    = \Psi_X(\mathcal{E}_H^{(n)}(f))(x).
  \end{align*}
  In other words, $\Psi_X \circ \mathcal{E}_H^{(n)} = (\mathcal{E}')^* \circ \Psi_X$ and hence $\Trace(\mathcal{E}_H^{(n)}) = \Trace(\mathcal{E}')$. To compute $\Trace(\mathcal{E}')$, fix a basis $\{ e_i \}$ of $H$. Let $\{ e_i^* \}$ denote the dual basis to $\{ e_i \}$. For simplicity of notation, we write $n - 1$ as $m$. By the definition of the trace,
  \begin{align*}
    \Trace(\mathcal{E}') & = \sum_{i_1, \dotsc, i_m} \Big \langle
    e_{i_1}^* \otimes \dotsm \otimes e_{i_m}^*,
    \mathcal{E}'(e_{i_1} \otimes \dotsb \otimes e_{i_m})
    \Big \rangle \\
    & = \sum_{i_1, \dotsc, i_{m}} \langle e_{i_1}^*, S(e_{i_1(m)}) e_{i_2} \rangle
    \dotsm \langle e_{i_{m-1}}^*, S(e_{i_1(2)}) e_{i_{m}} \rangle
    \langle e_{i_{m}}^*, S(e_{i_1(1)}) \rangle.
  \end{align*}
  By using the formula $\sum_{i} \langle e_i^*, x \rangle e_i = x$ repeatedly, we compute
  \begin{equation*}
    \Trace(\mathcal{E}')
    = \sum_i \langle e_i^*, S(e_{i(m)}) \dotsb S(e_{i(1)}) \rangle
    = \sum_i \langle e_i^*, S(e_{i}^{[n-1]}) \rangle
    = \nu_n(H). \qedhere
  \end{equation*}
\end{proof}

\begin{remark}
  Let $\mathcal{C}$ be a $k$-linear pivotal monoidal category with finite-\hspace{0pt}dimensional $\Hom$-spaces. Ng and Schauenburg \cite{MR2381536} defined the $n$-th Frobenius-Schur indicator of $V \in \mathcal{C}$ to be the trace of a certain linear operator
  \begin{equation}
    \label{eq:NS-FS-E-map}
    E_V^{(n)}: \Hom_{\mathcal{C}}(1_{\mathcal{C}}, V^{\otimes n}) \to \Hom_{\mathcal{C}}(1_{\mathcal{C}}, V^{\otimes n}),
  \end{equation}
  where $1_{\mathcal{C}}$ is the unit object of $\mathcal{C}$. Although $\Mod{H}_{\mathit{fd}}$ is not pivotal in general, there is always an isomorphism of $H$-modules
  \begin{equation*}
    \iota_H: H \to H^{**},
    \quad \langle \iota_H(h), f \rangle = \langle f, S^2(h) \rangle
    \quad (h \in H, f \in H^*).
  \end{equation*}
  Put $M = H^*$ and $j = (\iota_H^*)^{-1}: M \to M^{**}$. Treating $j$ as if it were a component of a pivotal structure of $\Mod{H}_{\mathit{fd}}$, we define an operator
  \begin{equation*}
    \tilde{E}_M^{(n)}: \Hom_H(k, M^{\otimes n}) \to \Hom_H(k, M^{\otimes n})
  \end{equation*}
  in the same way as we have defined~\eqref{eq:NS-FS-E-map}. By using graphical calculus as in \S5 of \cite{MR2381536}, we see that the canonical isomorphism
  \begin{equation*}
    \begin{CD}
      I_n(H) = \Hom_H(H^{\otimes n}, k)
      @>{(-)^*}>> \Hom_H(k^*, (H^{\otimes n})^*)
      @>{\cong}>> \Hom_H(k, M^{\otimes n})
    \end{CD}
  \end{equation*}
  makes the following diagram commute:
  \begin{equation*}
    \begin{CD}
      I_n(H) @. = \Hom_H(H^{\otimes n}, k) @>{\cong}>> \Hom_H(k, M^{\otimes n})\phantom{.} \\
      @V{\mathcal{E}_H^{(n)}}V{}V @. @V{}V{\tilde{E}_M^{(n)}}V \\
      I_n(H) @. = \Hom_H(H^{\otimes n}, k) @>>{\cong}> \Hom_H(k, M^{\otimes n}).
    \end{CD}
  \end{equation*}
  Since $M \cong H$ in $\Mod{H}_{\mathit{fd}}$, $\Trace(\tilde{E}_M^{(n)})$ can be thought as the $n$-th Frobenius-Schur indicator of $H$ with respect to ``the pivotal structure $\iota_H$''. Note that if $H$ is semisimple and $k$ is an algebraically closed field of characteristic zero, then the isomorphism $\iota_H$ is indeed a component of the canonical pivotal structure \cite{MR2183279}. The isomorphism $\iota_H$ does not seem to be ``canonical'' in general: Let $F$ be a twist of $H$ and consider the diagram
  \begin{equation*}
    \begin{CD}
      H^F @>{(\iota_H)_F}>> (H^{**})_F \\
      @V{\iota_{H_F}}V{}V @V{}V{\xi_{H^*}}V \\
      (H_F)^{**} @>>{\xi_H^*}> ((H^*)_F)^*,
    \end{CD}
  \end{equation*}
  where $(-)_F: \Mod{H} \to \Mod{H_F}$ is functor~\eqref{eq:gauge-eq-1} and $\xi_X: (X^*)_F \to (X_F)^*$ is the duality transformation \cite[\S1]{MR2381536}. If the above diagram would commute, we might say that ``the functor $(-)_F$ preserves $\iota_H$'' ({\it cf}. the definition of pivotal-structure-preserving functors \cite[\S1]{MR2381536}). Note that in our case $\xi_X$ is given by
  \begin{equation*}
    \langle \xi_X(f), x \rangle = \langle f, \beta_F^{-1} x \rangle \quad (f \in X^*, x \in X),
  \end{equation*}
  where $\beta_F$ is the element introduced in \S\ref{subsec:gauge-inv}. By using~\eqref{eq:gauge-eq-2}, one can check that the above diagram commute if and only if $\beta_F = S_H(\beta_F)$.
\end{remark}

\subsection{Basic properties of the indicators}
\label{subsec:ind-Hopf-basics}

By using Lemmas~\ref{lem:ind-integral} and~\ref{lem:ind-rot-map}, we derive basic properties of the indicators of finite-dimensional Hopf algebras. The following proposition is fundamental and easy to prove:

\begin{proposition}
  \label{prop:ind-basics}
  Let $H$ and $H'$ be finite-dimensional Hopf algebras. Then:
  \begin{enumerate}
  \item $\nu_n(H^*) = \nu_n(H)$ for all $n \ge 1$.
  \item $\nu_n(H \otimes H') = \nu_n(H) \cdot \nu_n(H')$ for all $n \ge 1$.
  \end{enumerate}
\end{proposition}
\begin{proof}
  (1) Let $\Lambda \in H$ and $\lambda \in H^*$ be left integrals such that $\langle \lambda, \Lambda \rangle = 1$. Since $\langle -, \Lambda \rangle$ is a left integral on $H^{*}$, we have
  \begin{equation*}
    \nu_n(H^*)
    = \langle \lambda^{[n]}, \Lambda \rangle
    = \langle \lambda_{(1)}, \Lambda_{(1)} \rangle
    \dotsb \langle \lambda_{(n)}, \Lambda_{(n)} \rangle
    = \langle \lambda, \Lambda^{[n]} \rangle
    = \nu_n(H)
  \end{equation*}
  by Lemma~\ref{lem:ind-integral} and the definition of the dual Hopf algebra.

  (2) Let $\Lambda' \in H'$ and $\lambda' \in H'{}^*$ be left integrals such that $\langle \lambda', \Lambda' \rangle = 1$. Then both $\tilde{\lambda} = \lambda \otimes \lambda' \in H^* \otimes H'{}^*$ and $\tilde{\Lambda} = \Lambda \otimes \Lambda' \in H \otimes H'$ are left integrals such that $\langle \tilde{\lambda}, \tilde{\Lambda} \rangle = 1$. Hence we have
  \begin{equation*}
    \nu_n(H \otimes H')
    = \langle \tilde{\lambda}, \tilde{\Lambda}^{[n]} \rangle
    = \langle \lambda, \Lambda^{[n]} \rangle \cdot \langle \lambda', \Lambda'{}^{[n]} \rangle
    = \nu_n(H) \cdot \nu_n(H')
  \end{equation*}
  by Lemma~\ref{lem:ind-integral} and the definition of the tensor product of Hopf algebras.
\end{proof}

By Lemma~\ref{lem:inv-cocycle-deform} and Proposition~\ref{prop:ind-basics} (1), we obtain:

\begin{corollary}
  \label{cor:ind-2-cocycle}
  $\nu_n$ is invariant under 2-cocycle deformation.
\end{corollary}

Fix an algebraic closure $\overline{k}$ of $k$. For each integer $N \ge 1$, we define  $\mathcal{O}_N$ to be the subring of $\overline{k}$ generated by the roots of the polynomial $X^N - 1$. By definition, each element of $\mathcal{O}_N$ is of the form $z = a_1 \omega_1 + \dotsb + a_m \omega_m$ for some $a_i \in \mathbb{Z}$ and some $\omega_i \in \overline{k}$ such that $\omega_i^N = 1$. For such an element $z$, we set
\begin{equation*}
  \overline{z} := a_1 \omega_1^{-1} + \dotsb + a_m \omega_m^{-1} \in \mathcal{O}_N
  \text{\quad and \quad}
  |z|^2 := z \cdot \overline{z}.
\end{equation*}
The assignment $z \mapsto \overline{z}$ is a well-defined ring automorphism of $\mathcal{O}_N$. Note that if $k \subset \mathbb{C}$ (and $\overline{k}$ is chosen so that $\overline{k} \subset \mathbb{C}$), then $\mathcal{O}_N$ is the ring of integers in the $N$-th cyclotomic field $\mathbb{Q}[e^{2\pi\mathbf{i}/N}]$ and $\overline{z}$ is the complex conjugate of $z$.

\begin{proposition}
  \label{prop:ind-cyc-int}
  Let $H$ be a finite-dimensional Hopf algebra. Then:
  \begin{enumerate}
  \item $\nu_n(H) \in k \cap \mathcal{O}_N$ for all $n \ge 1$, where $N = n \cdot \ord(S_H^2)$.
  \item $\nu_n(H^{\cop}) = \nu_n(H^{\op}) = \overline{\nu_n(H)}$ for all $n \ge 1$.
  \end{enumerate}
\end{proposition}
\begin{proof}
  (1) We use Lemma~\ref{lem:ind-rot-map} for the proof. Since
  \begin{equation*}
    (\mathcal{E}_H^{(n)}{})^n(f) = f \circ (S^2 \otimes \dotsb \otimes S^2)
  \end{equation*}
  for all $f \in I_n(H)$, $(\mathcal{E}_H^{(n)})^N$ is the identity map. Since the trace is the sum of the eigenvalues, we have $\nu_n(H) \in \mathcal{O}_N$. $\nu_n(H) \in k$ is trivial.

  (2) The first equality is clear since $S_H: H^{\op} \to H^{\cop}$ is an isomorphism of Hopf algebras. To prove the second, we define $\mathcal{R}: I_n(H) \to I_n(H^{\cop})$ by
  \begin{equation*}
    \mathcal{R}(f)(a_1, \dotsc, a_n) = f(a_n, \dotsc, a_1)
    \quad (f \in I_n(H), a_1, \dotsc, a_n \in H).
  \end{equation*}
  Fix $f \in I_n(H)$ and set $g = (\mathcal{R} \circ \mathcal{E}_H^{(n)})(f) \in I_n(H^{\cop})$. Then
  \begin{equation*}
    g(x_1, \dotsc, x_n) =  \mathcal{E}_H^{(n)}(f)(x_n, \dotsc, x_1) = f(x_{n - 1}, \dotsc, x_1, S_H^2(x_n))
  \end{equation*}
  for all $x_1, \dotsc, x_n \in H^{\cop}$. Since $S_{H^{\cop}} = S_H^{-1}$, we compute
  \begin{align*}
    (\mathcal{E}_{H^{\cop}}^{(n)}(g)) (a_1, \dotsc, a_n)
    = g(a_2, \dotsc, a_n, S_{H^{\cop}}^2(a_1))
    = f(a_n, \dotsc, a_1)
  \end{align*}
  for all $a_1, \dotsc, a_n \in H$. This means $\mathcal{E}_{H^{\cop}}^{(n)} \circ \mathcal{R} \circ \mathcal{E}_H^{(n)} = \mathcal{R}$. By Lemma~\ref{lem:ind-rot-map},
  \begin{equation*}
    \nu(H^{\cop}) = \Trace(\mathcal{E}_{H^{\cop}}^{(n)})
    = \Trace(\mathcal{R} \circ (\mathcal{E}_H^{(n)})^{-1} \circ \mathcal{R}^{-1})
    = \Trace((\mathcal{E}_H^{(n)})^{-1}).
  \end{equation*}
  Now let $\omega_1, \dotsc, \omega_m$ be the eigenvalues of $\mathcal{E}_H^{(n)}$. As we have seen in the proof of (1), $\omega_i$'s are roots of unity. Hence we obtain
  \begin{equation*}
    \nu_n(H^{\cop})
    = \Trace((\mathcal{E}_H^{(n)})^{-1})
    = \sum_{i} \omega_i^{-1}
    = \overline{\sum_{i} \omega_i}
    = \overline{\Trace(\mathcal{E}_H^{(n)})}
    = \overline{\nu_n(H)}. \qedhere
  \end{equation*}
\end{proof}

Suppose that $H$ is quasitriangular with universal R-matrix $R \in H \otimes H$. Then $R$ is a twist of $H$ and $H_R$ is isomorphic to $H^{\cop}$. Hence, by the gauge invariance and Proposition~\ref{prop:ind-cyc-int} (2), we obtain:

\begin{corollary}
  \label{cor:ind-Hopf-qt}
  $\nu_n(H) = \overline{\nu_n(H)}$ for all $n \ge 1$ if $H$ is quasitriangular.
\end{corollary}

Doi and Takeuchi \cite{MR1298746} showed that the Drinfeld double $D(H)$ can be obtained from $H^{*\cop} \otimes H$ by 2-cocycle deformation. Hence $\nu_n(D(H)) = \nu_n(H^{*\cop} \otimes H)$ by Corollary~\ref{cor:ind-2-cocycle}. By applying Propositions~\ref{prop:ind-basics} and~\ref{prop:ind-cyc-int}, we obtain:

\begin{corollary}
  \label{cor:ind-Drinfeld-double}
  $\nu_n(D(H)) = |\nu_n(H)|^2$ for all $n \ge 1$.
\end{corollary}

This formula has been conjectured in \cite{KMN09} and proved in the case where $H$ is a semisimple Hopf algebra over $\mathbb{C}$ by using the formula of the $n$-th Frobenius-Schur indicator of objects of a modular tensor category \cite[Theorem~5.6]{KMN09}.

\subsection{Indicators of filtered Hopf algebras}
\label{subsec:ind-Hopf-filtered}

Let $H$ be a finite-dimensional filtered Hopf algebra. Here we prove:

\begin{theorem}
  \label{thm:ind-filtered}
  $\nu_n(H) = \nu_n(\gr H)$ for all $n \ge 1$.
\end{theorem}

To prove this theorem, we understand the associated graded Hopf algebra in a categorical way: First, let $\FiltVect(k)$ and $\GrVect(k)$ be the category of filtered and graded vector spaces, respectively (recall our convention from \S\ref{subsec:notations}). They are symmetric monoidal categories and taking the associated graded vector space defines a $k$-linear symmetric strong monoidal functor
\begin{equation}
  \label{eq:functor-gr}
  \gr: \FiltVect(k) \to \GrVect(k).
\end{equation}
A filtered and a graded Hopf algebra are nothing but a Hopf algebra in $\FiltVect(k)$ and in $\GrVect(k)$, respectively. Now let $H$ be a filtered Hopf algebra. The above observation explains why $\gr(H)$ is a graded Hopf algebra: $\gr(H)$ is a graded Hopf algebra as an image of a Hopf algebra in $\FiltVect(k)$ under \eqref{eq:functor-gr}.

\begin{proof}[Proof of Theorem~\ref{thm:ind-filtered}]
  Since the claim is obvious for $n = 1$, we assume $n \ge 2$. For simplicity of notation, we put $m = n - 1$. By the above description of the associated graded Hopf algebra, we have
  \begin{align*}
    \begin{CD}
      \nabla^{(m)}_{\gr(H)} = \Big( \gr(H)^{\otimes m}
      @>{\varphi_{H,\dotsc,H}}>> \gr(H^{\otimes m})
      @>{\gr(\nabla_H^{(m)})}>> \gr(H) \Big),
    \end{CD} \\
    \begin{CD}
      \Delta^{(m)}_{\gr(H)} = \Big( \gr(H)
      @>{\gr(\Delta^{(m)}_H)}>> \gr(H^{\otimes m})
      @>{\varphi_{H,\dotsc,H}^{-1}}>> \gr(H)^{\otimes m} \Big),
    \end{CD}
  \end{align*}
  where $\nabla^{(m)}$ and $\Delta^{(m)}$ are given by \eqref{eq:iterated-mult} and \eqref{eq:iterated-comul}, respectively, and
  \begin{equation*}
    \varphi_{X_1, \dotsc, X_m}: \gr(X_1) \otimes \dotsb \otimes \gr(X_m) \to \gr(X_1 \otimes \dotsb \otimes X_m)
    \quad (X_i \in \FiltVect(k))
  \end{equation*}
  is the isomorphism obtained by using the monoidal structure of \eqref{eq:functor-gr}. By the definition of the Sweedler power map, we have
  \begin{equation*}
    P_{\gr(H)}^{(m)}
    = \nabla_{\gr(H)}^{(m)} \circ \Delta_{\gr(H)}^{(m)}
    = \gr(\nabla_H^{(m)}) \circ \gr(\Delta_H^{(m)})
    = \gr(P_H^{(m)})
  \end{equation*}
  and therefore
  \begin{equation*}
    \gr( S_H \circ P_H^{(m)})
    = \gr(S_H) \circ \gr(P_H^{(m)})
    = S_{\gr(H)} \circ P_{\gr(H)}^{(m)}.
  \end{equation*}
  If $f: X \to X$ is a morphism in $\FiltVect(k)$ such that $\dim_k(X) < \infty$, then we have $\Trace(f) = \Trace(\gr f)$. Thus, taking the trace of both sides of the above equation, we obtain $\nu_n(H) = \nu_n(\gr H)$.
\end{proof}

\subsection{Extending $\nu_n$ for $n \le 0$}
\label{subsec:ind-non-positive}

Since the $m$-th Sweedler power map is defined for all $m \in \mathbb{Z}$, we can define the $n$-th indicator $\nu_n$ by \eqref{eq:ind-def} for all $n \in \mathbb{Z}$. As we have remarked, $\nu_n$ is a gauge invariant for all $n \ge 1$ \cite{KMN09}. Here we first prove:

\begin{theorem}
  \label{thm:ind-gauge-inv}
  $\nu_n$ is a gauge invariant for all integer $n$.
\end{theorem}

Our proof relies on the gauge invariance of $\nu_n$ for $n \ge 1$ and the linear recurrence relation between the Sweedler power maps. Recall that a sequence $v = \{ v_n \}_{n \ge 1}$ of elements of a vector space is said to be {\em linearly recursive} if there exists a non-zero polynomial $\sum_{i = 0}^m c_i X^i \in k[X]$ such that
\begin{equation*}
  c_0 v_j + c_1 v_{j + 1} + \dotsb + c_m v_{j + m} = 0
\end{equation*}
for all $j \ge 1$. If this is the case, all such polynomials form a non-zero ideal of $k[X]$. The {\em minimal polynomial} of $v$ is the monic polynomial generating this ideal. Now we provide the following easy lemma:

\begin{lemma}
  \label{lem:min-pol-pow}
  Let $A$ be a finite-dimensional algebra, and let $a \in A$. Then the sequence $a^* = \{ a^n \}$ is linearly recursive. Let $f_a(X) \in k[X]$ be the minimal polynomial of the sequence $a^*$. Then we have:
  \begin{enumerate}
  \item $\deg f_a(X) \le \dim_k(A)$.
  \item If $g(X) \in k[X]$ satisfies $g(a) = 0$, then $f_a(X)$ divides $g(X)$.
  \item $f_a(0) \ne 0$ if and only if $a$ is invertible.
  \end{enumerate}
\end{lemma}
\begin{proof}
  Put $N = \dim_k(A)$. Since $\{ a^i \}_{i = 0}^N \subset A$ is linearly dependent, there exists a non-zero polynomial $\tilde{f}(X) = \sum_{i = 0}^{N} \lambda_i X^i \in k[X]$ such that $\tilde{f}(a) = 0$. Now let, in general, $g(X) = \sum_{i = 0}^m c_i X^i$ be a polynomial such that $g(a) = 0$. Then
  \begin{equation*}
    c_0 a^j + c_1 a^{j + 1} + \dotsb + c_m a^{j + m} = a^j \cdot g(a) = 0
  \end{equation*}
  for all $j \ge 0$. Applying this argument to $g(X) = \tilde{f}(X)$, we see that $a^*$ is linearly recursive and its minimal polynomial $f_a(X)$ divides $\tilde{f}(X)$. Since $\deg \tilde{f}(X) \le N$, we obtain (1). Again from the above argument, (2) follows.

  Now we prove (3). We can write $f_a(X) = f_a(0) + X \cdot p(X)$ for some $p(X) \in k[X]$. If $f_a(0) \ne 0$, then $a$ is invertible with $a^{-1} = -f_a(0)^{-1} p(a)$. If $f_a(0) = 0$, then $a \cdot p(a) = 0$. Since $p(a) \ne 0$ by the minimality of $f_a(X)$, $a$ is not invertible.
\end{proof}

Now we recall that $A = \End_k(H)$ is an algebra with respect to the convolution product. Applying the above lemma to $\id_H \in A$, we see that the sequence
\begin{equation*}
  P_H := \{ P_H^{(n - 1)} \}_{n \ge 1}
\end{equation*}
is linearly recursive ({\it cf.} the proof of Proposition 2.7 of \cite{KMN09}). Let $\Phi_H(X)$ be the minimal polynomial of $P_H$. We have $\Phi_H(0) \ne 0$ since $\id_H$ is invertible with respect to the convolution product (with the inverse $S_H$).

\begin{proof}[Proof of Proposition~\ref{thm:ind-gauge-inv}]
  It is sufficient to prove the gauge invariance of $\nu_n$ for all $n \le 0$. Thus, let $H_1$ and $H_2$ be finite-dimensional Hopf algebras and assume that they are gauge equivalent. In what follows, we show that
  \begin{equation}
    \label{eq:ind-gauge-inv-IH}
    \nu_{-n+1}(H_1) = \nu_{-n+1}(H_2)
  \end{equation}
  holds for all $n \ge 0$ by induction on $n$.

  If $n = 0$, then the claim is trivial. Now we suppose that \eqref{eq:ind-gauge-inv-IH} holds for all $0 \le n < \ell$. Let $\Phi(X)$ be the least common multiple of $\Phi_{H_1}(X)$ and $\Phi_{H_2}(X)$ and write it as $\Phi(X) = \sum_{i = 0}^m c_i X^i$. Since $c_0 = \Phi(0) \ne 0$ by Lemma~\ref{lem:min-pol-pow},
  \begin{equation*}
    P_H^{(0)} = -c_0^{-1} (c_1 P_H^{(1)} + \dotsb + c_m P_H^{(m)})
  \end{equation*}
  for $H = H_1, H_2$. Multiplying $P_H^{(-\ell)}$ to the both sides and applying
  \begin{equation}
    \label{eq:S-Tr-map}
    \End_k(H) \to k,
    \quad F \mapsto \Trace(S_H \circ F)
    \quad (F \in \End_k(H)),
  \end{equation}
  we get $\nu_{-\ell+1}(H) = -c_0^{-1} (c_1 \nu_{-\ell+2}(H) + \dotsb + c_m \nu_{-\ell+m+1}(H))$. By the gauge invariance of $\nu_n$ for $n \ge 1$ and the induction hypothesis, we obtain 
  \begin{align*}
    \nu_{-\ell+1}(H_1)
    & = -c_0^{-1} (c_1 \nu_{-\ell+2}(H_1) + \dotsb + c_m \nu_{-\ell+m+1}(H_1)) \\
    & = -c_0^{-1} (c_1 \nu_{-\ell+2}(H_2) + \dotsb + c_m \nu_{-\ell+m+1}(H_2))
    = \nu_{-\ell+1}(H_2).
  \end{align*}
  Hence \eqref{eq:ind-gauge-inv-IH} for $n = \ell$ is proved.
\end{proof}

Let $\lambda \in H^*$ and $\Lambda \in H$ be left integrals such that $\langle \lambda, \Lambda \rangle = 1$. Since the proof of Lemma~\ref{lem:ind-integral} is applicable even for $n \le 0$, we have $\nu_n(H) = \langle \lambda, \Lambda^{[n]} \rangle$ not only for $n \ge 1$ but also for $n \le 0$.

Among $\nu_n$'s for $n \le 0$, the values of $\nu_0$ and $\nu_{-1}$ are of special interest. First we compute $\nu_0$. By the definition and the above arguments, we have
\begin{equation*}
  \nu_0(H) = \Trace(S_H^2) = \langle \lambda, \Lambda^{[0]} \rangle = \langle \lambda, 1 \rangle \langle \varepsilon, \Lambda \rangle.
\end{equation*}
Hence, $\nu_0(H) \ne 0$ if and only if $H$ is semisimple and cosemisimple \cite{MR1265853}. Since such a Hopf algebra is involutive (see \cite{MR1643702} and \cite{MR926744}), we conclude:

\begin{proposition}
  \label{prop:ind-0}
  Let $H$ be a finite-dimensional Hopf algebra. Then:
  \begin{equation*}
    \nu_0(H) =
    \begin{cases}
      \dim_k(H) & \text{if $H$ is semisimple and cosemisimple}, \\
      0 & \text{otherwise}.
    \end{cases}
  \end{equation*}
\end{proposition}

For $n = -1$, the value of $\nu_{n}$ is described as follows:

\begin{proposition}
  \label{prop:ind-(-1)}
  Let $H$ be a finite-dimensional Hopf algebra, and let $\Lambda \in H$ be a non-zero left integral. Then $S^2_H(\Lambda) = \nu_{-1}(H) \Lambda$.
\end{proposition}
\begin{proof}
  Since the space of left integrals in $H$ is one-dimensional, $S^2(\Lambda) = c \Lambda$ for some $c \in k$. Let $\lambda_r \in H^*$ be a right integral such that $\langle \lambda_r, \Lambda \rangle = 1$. Since $\lambda = \lambda_r \circ S$ is a left integral such that $\langle \lambda, \Lambda \rangle = 1$ (see \cite{MR1265853}),
  \begin{equation*}
    \nu_{-1}(H) = \langle \lambda, \Lambda^{[-1]} \rangle
    = \langle \lambda_r, S^2(\Lambda) \rangle 
    = c \langle \lambda_r, \Lambda \rangle = c.
  \end{equation*}
  Hence $S^2(\Lambda) = c \Lambda = \nu_{-1}(H) \Lambda$.
\end{proof}

Let $H$ be a finite-dimensional Hopf algebra, and let $\Lambda \in H$ be a non-zero left integral. We note some properties of $\nu_{-1}(H)$. First, by Proposition~\ref{prop:ind-(-1)}, $\nu_{-1}(H)$ is an eigenvalue of $S^2_H$. In particular,
\begin{equation}
  \label{eq:ind-(-1)-1}
  \nu_{-1}(H)^{\ord(S^2_H)} = 1.
\end{equation}
$\Lambda$ is also a left integral in $H^{\cop}$. Again by Proposition~\ref{prop:ind-(-1)},
\begin{equation*}
  \nu_{-1}(H^{\cop}) \Lambda = S_{H^{\cop}}^2(\Lambda) = S_H^{-2}(\Lambda) = \nu_{-1}(H)^{-1} \Lambda.
\end{equation*}
This yields a generalization of Proposition~\ref{prop:ind-cyc-int} to $n = -1$, as follows:
\begin{equation}
  \label{eq:ind-(-1)-2}
  \nu_{-1}(H^{\cop}) = \overline{\nu_{-1}(H)}.
\end{equation}
By \eqref{eq:ind-(-1)-1} and~\eqref{eq:ind-(-1)-2}, in the same way as Corollary~\ref{cor:ind-Hopf-qt}, we obtain:

\begin{corollary}
  \label{cor:ind-(-1)-qt}
  $\nu_{-1}(H) \in \{\pm 1\}$ if $H$ is quasitriangular.
\end{corollary}

Let $\lambda \in H^*$ be a non-zero left integral. A {\em distinguished grouplike element} of $H$ is a unique grouplike element $a \in H$ such that $\langle \lambda \star f, h \rangle = \langle \lambda, h \rangle \langle f, a \rangle$ for all $f \in H^*$ and $h \in H$. Let $a \in H$ and $\alpha \in H^*$ be distinguished grouplike elements. Then we have $S^2(\Lambda) = \langle \alpha, a \rangle \Lambda$ (see \cite{MR1265853}) and hence
\begin{equation}
  \label{eq:ind-(-1)-3}
  \nu_{-1}(H) = \langle \alpha, a \rangle.
\end{equation}
Recall that $H$ is said to be {\em unimodular} if $\Lambda$ is central. If this is the case, then the distinguished grouplike element of $H^*$ is the counit $\varepsilon \in H^*$. Since $\varepsilon(g) = 1$ for all grouplike element $g \in H$, we have:

\begin{corollary}
  \label{cor:ind-(-1)-unimodular}
  $\nu_{-1}(H) = 1$ if $H$ is unimodular.
\end{corollary}

For $n > 1$, we can express the value of $\nu_{-n}$ by using $\nu_{-1}$ and $\nu_n$.

\begin{proposition}
  \label{prop:ind-(-n)}
  Let $H$ be a finite-dimensional Hopf algebra. Then:
  \begin{enumerate}
  \item $\nu_{-n}(H) = \nu_{-1}(H) \cdot \overline{\nu_n(H)}$ for all $n \ge 1$.
  \item If $H$ is unimodular, then $\nu_{-n}(H) = \overline{\nu_n(H)}$ for all $n \ge 1$.
  \end{enumerate}
\end{proposition}
\begin{proof}
  (1) Let $\lambda_r \in H^*$ be a non-zero right integral, and let $\Lambda \in H$ be a left integral such that $\langle \lambda_r, \Lambda \rangle = 1$. Since $\lambda = \lambda_r \circ S$ is a left integral such that $\langle \lambda, \Lambda \rangle = 1$,
  \begin{equation*}
    \nu_{-n}(H) = \langle \lambda, \Lambda^{[-n]} \rangle
    = \langle \lambda_r, S^2(\Lambda_{(n)}) \dotsb S^2(\Lambda_{(1)}) \rangle.
  \end{equation*}
  Here we note that, by Proposition~\ref{prop:ind-(-1)},
  \begin{equation*}
    S^2(\Lambda_{(n)}) \dotsb S^2(\Lambda_{(1)}) = S^2(\Lambda)_{(n)} \dotsb S^2(\Lambda)_{(1)} = \nu_{-1}(H) \cdot \Lambda_{(n)} \dotsb \Lambda_{(1)}.
  \end{equation*}
  Note also that $\lambda_r \in (H^{\cop})^*$ and $\Lambda \in H^{\cop}$ are left integrals such that $\langle \lambda_r, \Lambda \rangle = 1$. Hence, by Lemma~\ref{lem:ind-integral} and Proposition~\ref{prop:ind-cyc-int},
  \begin{equation*}
    \overline{\nu_n(H)} = \nu_n(H^{\cop}) = \langle \lambda_r, \Lambda_{(n)} \dotsm \Lambda_{(1)} \rangle.
  \end{equation*}
 Combining the above results, we obtain
  \begin{equation*}
    \nu_{-n}(H)
    = \nu_{-1}(H) \langle \lambda_r, \Lambda_{(n)} \dotsb \Lambda_{(1)} \rangle
    = \nu_{-1}(H) \cdot \overline{\nu_n(H)}.
  \end{equation*}
  (2) This follows from (1) and Corollary~\ref{cor:ind-(-1)-unimodular}.
\end{proof}

Hence, in principle, the properties of $\nu_{-n}$ for $n \ge 1$ can be deduced from the properties of $\nu_n$ and $\nu_{-1}$. The previous results on $\nu_n$ for $n \ge 1$ can be extended as follows:

\begin{corollary}
  \label{cor:ind-basic-all-n}
  Let $H$ and $H'$ be finite-dimensional Hopf algebras. Then, for all $n \in \mathbb{Z}$, we have the following:
  \begin{enumerate}
  \item $\nu_n(H^*) = \nu_n(H)$ and $\nu_n(H \otimes H') = \nu_n(H) \cdot \nu_n(H')$.
  \item $\nu_n(H) \in k \cap \mathcal{O}_N$, where $N = n \cdot \ord(S_H^2)$.
  \item $\nu_n(H^{\op}) = \nu_n(H^{\cop}) = \overline{\nu_n(H)}$.
  \item $\nu_n(D(H)) = |\nu_n(H)|^2$.
  \item $\nu_n(H) = \overline{\nu_n(H)}$ if $H$ is quasitriangular.
  \item $\nu_n(H) = \nu_{-n}(H)$ if $H$ is unimodular.
  \item If $H$ is filtered, then $\nu_n(\gr(H)) = \nu_n(H)$.
  \end{enumerate}
  See \S\ref{subsec:ind-Hopf-basics} for the definition of $\mathcal{O}_m \subset \overline{k}$ for $m > 0$. For $m = 0$, we define $\mathcal{O}_m$ to be the image of $\mathbb{Z} \to k$, $n \mapsto n 1$. For $m < 0$, we set $\mathcal{O}_m =\mathcal{O}_{-m}$.
\end{corollary}
\begin{proof}
  (1)--(6) can be obtained from the results of \S\ref{subsec:ind-Hopf-basics} by using  \eqref{eq:ind-(-1)-1}, \eqref{eq:ind-(-1)-2}, \eqref{eq:ind-(-1)-3} and Proposition~\ref{prop:ind-(-n)}. To prove (7), show $P_{\gr(H)}^{(m)} = \gr(P_H^{(m)})$ for all $m \le 0$ in the same way as the proof of Theorem~\ref{thm:ind-filtered}.
\end{proof}

\subsection{Applications to the Taft algebras}

Let $N > 1$ be an integer such that $N \ne 0$ in $k$, and suppose that there exists a primitive $N$-th root $\omega \in k$ of unity. The {\em Taft algebra} $T_{N^2}(\omega)$ is a Hopf algebra generated as an algebra by $x$ and $g$ with relations $x^N = 0$, $g^N = 1$, and $g x = \omega x g$. The coalgebra structure is given by
\begin{equation}
  \label{eq:taft-coalg}
  \Delta(x) = x \otimes 1 + g \otimes x, \quad
  \varepsilon(x) = 0, \quad
  \Delta(g) = g \otimes g, \text{\quad and \quad}
  \varepsilon(g) = 1,
\end{equation}
and the antipode is determined by $S(x) = -g^{-1} x$ and $S(g) = g^{-1}$. It is easy to see that $\Lambda = (1 + g + \dotsb + g^{N - 1}) x^{N - 1}$ is a non-zero left integral. Since
\begin{equation*}
  S^2(\Lambda) = (1 + g + \dotsb + g^{N - 1}) \cdot S^2(x)^{N - 1} = \omega^{-N+1} \Lambda = \omega \Lambda,
\end{equation*}
we get $\nu_{-1}(T_{N^2}(\omega)) = \omega$ by Proposition~\ref{prop:ind-(-1)}. By the gauge invariance of $\nu_{-1}$, we obtain the following result, which is proved in \cite{KMN09} by computing the second indicator of $T_{N^2}(\omega)$.

\begin{theorem}
  Let $\omega$ and $\omega'$ be primitive $N$-th roots of unity. $T_{N^2}(\omega)$ and $T_{N^2}(\omega')$ are gauge equivalent if and only if $\omega = \omega'$.
\end{theorem}

By Corollary~\ref{cor:ind-(-1)-qt}, we also obtain the following result:

\begin{theorem}
  $T_{N^2}(\omega)$ is not quasitriangular if $N > 2$.
\end{theorem}

It has been well-known when the Taft algebra is quasitriangular. We note that $T_4(-1)$ (over a field $k$ such that $\mathrm{char}(k) \ne 2$) has a family of universal R-matrices $R_\alpha$ parametrized by $\alpha \in k$ \cite{MR1220770}.

\section{Minimal polynomial of the indicators}
\label{sec:min-pol-ind}

\subsection{Roots of the minimal polynomial}
\label{subsec:roots-min-pol}

Let $H$ be a finite-dimensional Hopf algebra and consider the sequences
\begin{equation*}
  P_H := \{ P_H^{(n - 1)} \}_{n \ge 1} \subset \End_k(H)
  \text{\quad and \quad}
  \nu_H := \{ \nu_n(H) \}_{n \ge 1} \subset k.
\end{equation*}
As we have seen in \S\ref{subsec:ind-non-positive}, the former sequence is linearly recursive. Since $\nu_H$ is the image of $P_H$ under \eqref{eq:S-Tr-map}, $\nu_H$ is also linearly recursive \cite[Proposition~2.7]{KMN09}. Now let $\Phi_H(X)$ and $\phi_H(X)$ be the minimal polynomial of $P_H$ and $\nu_H$, respectively. By definition,
\begin{equation}
  \label{eq:min-pol-Phi}
  \phi_H(X) \mid \Phi_H(X).
\end{equation}
The aim of this section is to study the properties of the indicators via $\phi_H(X)$. One of the main results of this section is the following theorem:

\begin{theorem}
  \label{thm:min-pol-roots}
  Every root of $\phi_H(X)$ is a root of unity.
\end{theorem}

At this time, we do not know any direct method to treat $\phi_H(X)$ itself. We rather prove that every root of $\Phi_H(X)$ is a root of unity. Once such a claim is proved, then Theorem~\ref{thm:min-pol-roots} follows immediately from \eqref{eq:min-pol-Phi}.

Now we provide some lemmas to prove Theorem~\ref{thm:min-pol-roots}. The first lemma will be used in \S\ref{sec:higher-ind-taft} for computation of the indicators of the Taft algebras and $u_q(\mathfrak{sl}_2)$. The proof is straightforward and omitted.

\begin{lemma}
  \label{lem:Sw-pow-map}
  Let $H$ be a Hopf algebra, and let $x$ and $g$ be elements of $H$. If $g$ is a grouplike element, then we have
  \begin{equation*}
    (x g)^{[n]}
    = x_{(1)} \varphi(x_{(2)}) \dotsb \varphi^{n-1}(x_{(n)}) \cdot g^n
  \end{equation*}
  for all $n \ge 1$, where $\varphi(a) = g a g^{-1}$ for all $a \in H$.
\end{lemma}

The proof of the following lemma is also straightforward and omitted:

\begin{lemma}
  \label{lem:min-pol-sub}
  Let $H$ be a finite-dimensional Hopf algebra. If $H'$ is a Hopf subalgebra of $H$, then $\Phi_{H'}(X)$ divides $\Phi_{H}(X)$.
\end{lemma}

As an aside, $\phi_{H'}(X)$ may not divide $\phi_{H}(X)$. For example, suppose $\mathrm{char}(k) \ne 2$ and take $H = T_4(-1)$. Then $H' = k (\mathbb{Z}/2\mathbb{Z})$ is a Hopf subalgebra of $H$ but $\phi_{H'}(X) = X^2 - 1$ does not divide $\phi_H(X) = (X - 1)^2$ (see \cite[Example~3.4]{KMN09}).

\begin{proof}[Proof of Theorem~\ref{thm:min-pol-roots}]
  Recall that a Hopf algebra $A$ is said to be {\em pivotal} if it has a grouplike element $g \in H$ such that $S^2(a) = g a g^{-1}$ for all $a \in A$. Such an element $g$ is referred as a {\em pivotal element} of $A$.

  {\bf Step 1} (Pivotal case). Suppose that $H$ has a pivotal element $g$. Define $R_g: H \to H$ by $R_g(x) = x g$ for $x \in H$. Since $g^{-1} x g = S^{-2}(x)$, we have
  \begin{equation*}
    (x g^{-1})^{[n]} = x_{(1)} S^{-2}(x_{(2)}) \dotsb S^{-2(n-1)}(x_{(n)}) \cdot g^{-n}
  \end{equation*}
  for all $n \ge 1$ by Lemma~\ref{lem:Sw-pow-map}. In other words,
  \begin{equation}
    \label{eq:min-pol-roots-1}
    P_H^{(n)}\circ R_g^{-1} = R_g^{-n} \circ T_H^{(n)}
  \end{equation}
  for all $n \ge 1$ with the notation in \S\ref{subsec:exponent}. Note that \eqref{eq:min-pol-roots-1} holds also for $n = 0$. We put $n = \lcm(\qexp(H), \ord(g))$. Then $\sum_{j = 0}^m (-1)^j \binom{m}{j} T_{H}^{(n j)} = 0$ for some $m > 0$ by the definition of the quasi-exponent. Since $R_g^n = \id_H$, we obtain
  \begin{equation*}
    \sum_{j = 0}^m (-1)^j \binom{m}{j} P_H^{(n j)}
    = \sum_{j = 0}^m (-1)^j \binom{m}{j} R_g^{-n j} \circ T_{H}^{(n j)} \circ R_g = 0
  \end{equation*}
  by~\eqref{eq:min-pol-roots-1}. By Lemma~\ref{lem:min-pol-pow} (2), this implies that $\Phi_H(X)$ divides $(X^n - 1)^m$, and therefore every root of $\Phi_H(X)$ is an $n$-th root of unity.

  {\bf Step 2} (General case). Now we consider the general case. Let $G$ be the subgroup of the group of Hopf algebra automorphisms of $H$ generated by $g := S_H^2$. Since the semidirect product $L = H \rtimes G$ is a pivotal Hopf algebra with pivotal element $g$, by Step 1, every root of $\Phi_L(X)$ is a root of unity. By Lemma~\ref{lem:min-pol-sub}, every root of $\Phi_H(X)$ is also a root of unity.
\end{proof}

\subsection{Behavior of the indicator sequence}
\label{subsec:ind-behavior}

The behavior of a linearly recursive sequence is dominated by its minimal polynomial. Theorem~\ref{thm:min-pol-roots} yields some important consequences on the behavior of the sequence $\nu_H = \{ \nu_n(H) \}_{n \ge 1}$. To describe our result, we fix an algebraic closure $\overline{k}$ of $k$. For a finite-dimensional Hopf algebra $H$, factorize $\phi_H(X)$ as
\begin{equation}
  \label{eq:factorize-phi}
  \phi_H(X) = \prod_{i = 1}^r (X - \omega_i)^{m_i},
  \quad m_i \ge 1,
  \quad \omega_i \ne \omega_j
  \quad (i \ne j)
\end{equation}
in $\overline{k}[X]$ and then define
\begin{equation*}
  e(H) = \lcm(\ord(\omega_1), \dotsc, \ord(\omega_r)) \text{\quad and \quad} m(H) = \max(m_1, \dotsc, m_r).
\end{equation*}

\begin{theorem}
  \label{thm:ind-behavior}
  For each finite-dimensional Hopf algebra $H$, there exists a unique series of periodic sequences of elements of $k$,
  \begin{equation*}
    c_j = \{ c_j(n) \}_{n \ge 1} \quad (j = 0, 1, 2, \dotsc),
  \end{equation*}
  satisfying the following three conditions:
  \begin{itemize}
  \item[(C1)] The period of each $c_j$ is non-zero in $k$.
  \item[(C2)] $c_j = 0$ for all sufficiently large $j$.
  \item[(C3)] For all $n \ge 1$, one has $\nu_n(H) = \sum_{j = 0}^{\infty} \binom{n}{j} c_j(n)$.
  \end{itemize}
  Moreover, such $c_j$'s satisfy the following two conditions:
  \begin{itemize}
  \item[(C4)] The period of each $c_j$ divides $e(H)$.
  \item[(C5)] $c_j = 0$ if $j \ge m(H)$.
  \end{itemize}
\end{theorem}

To prove this theorem, we require:

\begin{lemma}
  \label{lem:det-binom}
  Fix non-negative integers $a$, $b$ and $r$. If $a_{i j} = \binom{a + b i}{j}$, then
  \begin{equation*}
    \det \Big[ ( a_{i j} )_{i, j = 0, \dotsc, r} \Big] = b^{\frac{1}{2}r(r-1)}
  \end{equation*}
\end{lemma}
\begin{proof}
  Let $P_0(X), \dotsc, P_r(X) \in \mathbb{Q}[X]$ be polynomials of the form $P_j(X) = a_j X^j + $ (lower degree terms). Then, by a generalization of the Vandermonde determinant \cite[Proposition~1]{MR1701596}, we have
  \begin{equation*}
    \det \Big[ (P_j(X_i))_{i, j = 0, \dotsc, r} \Big]
    = a_0 a_1 \dotsm a_r \prod_{0 \le i < j \le r} (X_j - X_i)
  \end{equation*}
  in $\mathbb{Q}[X_1, \dotsc, X_n]$. Considering the case where $P_0(X) = 1$,
  \begin{equation*}
    P_j(X) = \frac{X (X - 1) \dotsm (X - j + 1)}{j!}
    \quad (j = 1, \dotsc, r),
  \end{equation*}
  and $X_i = a + b i$ ($i = 0, \dotsc, r$), we obtain the desired formula.
\end{proof}

\begin{proof}[Proof of Theorem~\ref{thm:ind-behavior}]
  Factorize $\phi_H(X)$ as in \eqref{eq:factorize-phi}. Since $\nu_H = \{ \nu_n(H) \}_{n \ge 1}$ is a solution of the linear recurrence equation corresponding to $\phi_H(X)$, there exists a family of elements $c_{i j} \in \overline{k}$ ($1 \le i \le r$, $0 \le j < m_i$) such that
  \begin{equation*}
    \nu_n(H) = \sum_{i = 1}^r \sum_{j = 0}^{m_i - 1} c_{i j} \omega_i^n \binom{n}{j}
  \end{equation*}
  for all $n \ge 1$ (see, {\em e.g.}, \cite{MR1990179}). For convenience, set $c_{i j} = 0$ for $j \ge m_i$. Then define $c_j = \{ c_j(n) \}_{n \ge 1}$ by $c_j(n) = \sum_{i = 0}^r c_{i j} \omega_i^n$ for $j \ge 0$ and $n \ge 1$. It is clear that $c_j$'s satisfy the required conditions (C1)-(C5), except that each $c_j$ is indeed a sequence of elements of $k$.

  Now we show that $c_j(n) \in k$ for all $j \ge 0$ and $n \ge 1$. Set $m = m(H)$ and $e = e(H)$. Since $c_j = 0$ for $j \ge m$, we only consider $c_0, \dotsc, c_{m-1}$. Since they are periodic with period dividing $e$, we have a system of linear equations
  \begin{equation*}
    \sum_{j = 0}^{m - 1} \binom{n + i e}{j} x_j
    = y_i \quad (i = 0, \dotsc, m - 1)
  \end{equation*}
  with $x_j = c_j(n)$ and $y_i = \nu_{n + i e}(H)$. Let $A$ be the matrix of coefficients of this system. By Lemma~\ref{lem:det-binom}, $\det(A) = e^{\frac{1}{2}(m - 1)(m - 2)} \ne 0$ in $k$. Solving the above system of equations, we see
  \begin{equation*}
    c_j(n) \in e^{-\frac{1}{2} (m-1)(m-2)}
    \cdot \mathrm{span}_{\mathbb{Z}} \{ \nu_{n + i e}(H) \mid i = 0, \dotsc, m - 1 \}
    \subset k.
  \end{equation*}

  The uniqueness of $\{ c_j \}$ follows from the uniqueness of the solution of a similar system of equations. Suppose that both $\{ c_j' \}_{j \ge 0}$ and $\{ c_j'' \}_{j \ge 0}$ satisfy the conditions (C1)-(C3). By (C2), there exists $M > 0$ such that $c_j' = c_j'' = 0$ for all $j \ge M$. Let $E$ be the least common multiple of the periods of $c_j'$ and $c_j''$ for $j = 0, \dotsc, M - 1$. 
  By (C3), we have a system of linear equations
  \begin{equation*}
    \sum_{j = 0}^M \binom{n + i E}{j} x_j = 0 \quad (i = 0, \dotsc, M - 1)
  \end{equation*}
  with $x_j = c_j'(n) - c_j''(n)$. Note that $E \ne 0$ in $k$ by (C1). By Lemma~\ref{lem:det-binom}, this system has a unique solution and therefore $c_j' = c_j''$ for all $j \ge 0$.
\end{proof}

Let $p$ be a prime number. Lucas' theorem implies that the sequence $\{ \binom{n}{j} \}_{n \ge 1}$ is periodic modulo $p$. Hence, by Theorem~\ref{thm:ind-behavior}, we have:

\begin{corollary}
  Suppose $\mathrm{char}(k) > 0$. Then, for any finite-dimensional Hopf algebra $H$ over $k$, the sequence $\{ \nu_n(H) \}_{n \ge 1}$ is periodic.
\end{corollary}

If $\mathrm{char}(k) = 0$, then $\binom{n}{j}$ is expressed as a polynomial of $n$. Hence:

\begin{corollary}
  \label{cor:ind-quasi-pol}
  Suppose $\mathrm{char}(k) = 0$. Then, for a finite-dimensional Hopf algebra $H$ over $k$, there exist finitely many periodic sequences $\tilde{c}_0, \dotsc, \tilde{c}_{\ell-1}$ such that
  \begin{equation*}
    \nu_n(H) = \tilde{c}_0(n) + \tilde{c}_1(n) \cdot n + \dotsb + \tilde{c}_{\ell - 1}(n) \cdot n^{\ell - 1}
  \end{equation*}
  for all $n \ge 1$.
\end{corollary}

If $k = \mathbb{C}$ is the field of complex numbers, then we can discuss the asymptotic behavior of $\nu_n(H)$ as $n \to \infty$. By Corollary~\ref{cor:ind-quasi-pol}, we see that (the absolute value of) $\nu_n(H)$ is bounded above by a polynomial of $n$. Namely, we have:

\begin{corollary}
  If $H$ is a finite-dimensional Hopf algebra over $\mathbb{C}$, then there exists a non-negative integer $\ell$ such that $\lim_{n \to \infty} n^{-\ell} \nu_n(H) = 0$.
\end{corollary}

\subsection{Hopf algebra with a pivotal element}

Let $H$ be a finite-dimensional Hopf algebra. In view of Theorem \ref{thm:ind-behavior}, it is important to know $e(H)$ and $m(H)$ to study the behavior of the sequence $\nu_H$.

In \cite{KMN09}, they observed that there seems to be some relations between $\phi_H(X)$ and $\qexp(H)$. Refining the proof of Theorem~\ref{thm:min-pol-roots}, we obtain the following relation between $e(H)$ and $\qexp(H)$ under the assumption that $H$ has a pivotal element.

\begin{proposition}
  If $H$ is a finite-dimensional pivotal Hopf algebra, then
  \begin{equation}
    \label{eq:e-divides-qexp}
    e(H) \mid \qexp(H).
  \end{equation}
\end{proposition}
\begin{proof}
  Let $g \in H$ be a pivotal element. In the proof of Theorem~\ref{thm:min-pol-roots}, we have proved that $e(H)$ divides $n := \lcm(\qexp(H), \ord(g))$. Hence we obtain
  \begin{equation*}
    e(H) = e(H)_p' \mid n_p' = \lcm(\qexp(H)_p', \ord(g)_p') = \qexp(H),
  \end{equation*}
  where $p = \mathrm{char}(k)$, by~\eqref{eq:qexp-char-p} and \eqref{eq:qexp-grplike-ord} with convention $n_0' = n$.
\end{proof}

\begin{remark}
  Let $H$ be an arbitrary finite-dimensional Hopf algebra. By the above proposition and the proof of Theorem~\ref{thm:min-pol-roots}, $e(H)$ divides $\qexp(H \rtimes G)$ with $G = \langle S_H^2 \rangle$. Thus it could be an interesting problem to express $\qexp(H \rtimes G)$ in a familiar way.
\end{remark}

\subsection{Hopf algebra with the dual Chevalley property}

We give estimations of $e(H)$ and $m(H)$ for the case where the coradical of $H$ is an involutely Hopf subalgebra.

\begin{lemma}
  \label{lem:min-pol-dual-Ch}
  Let $H$ be a finite-dimensional Hopf algebra satisfying the following two conditions:
  \begin{enumerate}
  \item[(DC1)] $H$ has the dual Chevalley property.
  \item[(DC2)] $S^2(h) = h$ for all $h \in H_0 := \corad(H)$.
  \end{enumerate}
  Then $\Phi_H(X)$ divides $(X^e - 1)^\ell$, where $e = \exp(H_0)$ and $\ell = \Lw(H)$.
\end{lemma}

By \eqref{eq:min-pol-Phi}, $\phi_H(X) \mid (X^e - 1)^\ell$. Hence, by \eqref{eq:qexp-char-0} and \eqref{eq:qexp-char-p}, $e(H) \mid \qexp(H_0)$. Since $H_0$ is a Hopf subalgebra of $H$, $\qexp(H_0) \mid \qexp(H)$. In conclusion, \eqref{eq:e-divides-qexp} holds under the above conditions.

Before we give a proof, we note that (DC2) is redundant if $\mathrm{char}(k) = 0$; indeed, then the cosemisimplicity of $H_0$ implies (DC2) by the results of \cite{MR957441,MR926744}. In the case of $\mathrm{char}(k) > 0$, it is conjectured, but not proved, that all finite-dimensional cosemisimple Hopf algebras are involutive (Kaplansky's fifth conjecture); see \cite{MR1643702,MR1623961} for partial results on this conjecture.

\begin{proof}
  Let $f = P_H^{(e)} - P_H^{(0)}$. By (DC2), $T_H^{(e)}|_{H_0} = P_H^{(e)}|_{H_0}$, where $T_H^{(e)}$ is given by~\eqref{eq:modified-Sw-pow}. Hence $f|_{H_0} = 0$ by the definition of the exponent. By Lemma~\ref{lem:nilpo},
  \begin{equation*}
    \sum_{j = 0}^\ell \binom{\ell}{j} (-1)^j P_H^{(e j)}
    = (P_H^{(e)} - P_H^{(0)})^{\star \ell}
    = f^{\star \ell}
    = 0.
  \end{equation*}
  By Lemma~\ref{lem:min-pol-pow} (2), we conclude that $\Phi_H(X)$ divides $(X^e - 1)^\ell$.
\end{proof}

This lemma can be applied to pointed Hopf algebras. Given a Hopf algebra $H$, we denote by $G(H)$ the group of grouplike elements of a Hopf algebra $H$. Note that, if $H$ is pointed, then
\begin{equation*}
  \exp(\corad(H)) = \exp(G(H)) \quad (= \min \{ n \ge 1 \mid \text{$g^n = 1$ for all $g \in G(H)$} \})
\end{equation*}
since $\corad(H)$ is the group algebra of $G(H)$. Combining Lemma~\ref{lem:min-pol-dual-Ch} and the results of \S\ref{subsec:ind-behavior}, we obtain:

\begin{corollary}
  \label{cor:ind-beh-ptd-ch-0}
  Suppose $\mathrm{char}(k) = 0$. If $H$ is a finite-dimensional pointed Hopf algebra, there are sequences $\tilde{c}_0, \dotsc, \tilde{c}_{\ell - 1}$ of elements of $k$, where $\ell = \Lw(H)$, such that each $\tilde{c}_j$ is periodic with period $\exp(G(H))$ and
  \begin{equation*}
    \nu_n(H) = \tilde{c}_0(n) + \tilde{c}_1(n) \cdot n + \dotsb + \tilde{c}_{\ell - 1}(n) \cdot n^{\ell - 1}
  \end{equation*}
  holds for all $n \ge 1$.
\end{corollary}

Suppose $p := \mathrm{char}(k) > 0$. For an integer $n > 0$, let $n_p$ denote the largest power of $p$ dividing $n$ (or $n_p = (n_p')^{-1} \cdot n$ with the notation of \S\ref{subsec:exponent}). Now let $H$ be a finite-dimensional pointed Hopf algebra and set
\begin{equation*}
  e = \exp(G(H)), \quad \ell = \Lw(H), \quad f = e_p' \text{\quad and \quad} m = e_p \ell.
\end{equation*}

\begin{corollary}
  \label{cor:ind-beh-ptd-ch-p}
  Under the above assumptions, there are sequences $c_0, \dotsc, c_{m - 1}$ such that each $c_j$ is periodic with period $f$ and
  \begin{equation*}
    \nu_n(H) = c_0(n) + \binom{n}{1} c_1(n) + \dotsb + \binom{n}{m - 1} c_{m - 1}(n)
  \end{equation*}
  holds for all $n \ge 1$. In particular, if $G(H)$ is a $p$-group, then $c_j$'s are constant.
\end{corollary}
\begin{proof}
  This follows from $(X^e - 1)^\ell = (X^f - 1)^m$.
\end{proof}

\begin{remark}
  Let $H$ be a finite-dimensional Hopf algebra. In \cite{KMN09},
  \begin{equation}
    \label{eq:min-pol-bound-dim-2}
    \deg \Phi_H(X) \le \dim_k(H)^2
  \end{equation}
  has been shown without any assumptions on $H$ ({\em cf}. Lemma~\ref{lem:min-pol-pow} (1)). Now suppose that (DC1) and (DC2) are satisfied. Then, by Lemmas \ref{lem:ll-bound} and \ref{lem:min-pol-dual-Ch},
  \begin{equation*}
    \deg \phi_H(X) \le \deg \Phi_H(X) \le \exp(H_0) \Lw(H) \le \frac{\exp(H_0) \dim_k(H)}{\dim_k(H_0)},
  \end{equation*}
  where $H_0 = \corad(H)$. Hence, if $H$ is pointed, then we obtain
  \begin{equation}
    \label{eq:min-pol-bound-dim}
    \deg \Phi_H(X) \le \dim_k(H)
  \end{equation}
  since $\exp(H_0) \mid \dim_k(H_0)$. This can be considered as a refinement of \eqref{eq:min-pol-bound-dim-2} for the case where $H$ is pointed. It is interesting to know when \eqref{eq:min-pol-bound-dim} holds. A positive answer to Kashina's conjecture (see \S\ref{subsec:exponent}) would imply that \eqref{eq:min-pol-bound-dim} holds for all finite-dimensional Hopf algebras over a field of characteristic zero having the dual Chevalley property.
\end{remark}

\section{Applications to a family of pointed Hopf algebras}
\label{sec:appl-pointed}

\subsection{The Hopf algebra $u(\mathcal{D}, \lambda, \mu)$}

Throughout, the base field $k$ is assumed to be an algebraically closed field of characteristic zero. In a series of papers \cite{MR1659895,MR1780094,MR1886004,MR1913436,MR2108213,MR2630042}, Andruskiewitsch and Schneider have classified all finite-dimensional pointed Hopf algebras $H$ such that $G(H)$ is abelian and all prime divisors of $|G(H)|$ are greater than $7$. As a result, such a Hopf algebra is isomorphic to the Hopf algebra $u(\mathcal{D}, \lambda, \mu)$, where $\mathcal{D}$, $\lambda$ and $\mu$ are certain parameters.

Following mainly \cite{MR2630042}, we recall the construction of $u(\mathcal{D}, \lambda, \mu)$. Let $\Gamma$ be a finite abelian group, and let $\Gamma^\vee$ be the group of characters of $\Gamma$. The first parameter
\begin{equation}
  \label{eq:datum-0}
  \mathcal{D} = (\Gamma, (g_i)_{i \in I}, (\chi_i)_{i \in I}, A = (a_{i j})_{i, j \in I})
\end{equation}
is called a {\em datum of finite Cartan type}. Here, $(g_i)_{i \in I}$ and $(\chi_i)_{i \in I}$ are families of elements of $\Gamma$ and $\Gamma^\vee$, respectively, indexed by a totally ordered finite set $I$, and $A$ is a Cartan matrix of finite type. The datum $\mathcal{D}$ is required to satisfy the following three conditions:
\begin{enumerate}
\item [(1)] $q_{i j} q_{j i} = q_{i i}^{a_{i j}}$ and $q_{i i} \ne 1$ for all $i, j \in I$, where $q_{i j} = \chi_j(g_i)$.
\item[(2)] $N_i = \ord(q_{i i})$ is odd for all $i \in I$.
\item[(3)] $N_i \not \equiv 0 \pmod{3}$ if $i$ is in a connected component of type $G_2$.
\end{enumerate}
We write $i \sim j$ if $i$ and $j$ are in the same connected component in the Dynkin diagram of $A$. By the condition (1), we have
\begin{equation*}
  q_{i j} q_{j i} = 1 \text{\quad for all $i, j \in I$ such that $i \not \sim j$}.
\end{equation*}
From (1), $q_{i i}^{a_{i j}} = q_{j j}^{a_{j i}}$ follows. By the conditions (2) and (3), we have
\begin{equation}
  \label{eq:datum-3}
  N_i = N_j \text{\quad if $i \sim j$}.
\end{equation}
We say that $i, j \in I$ are {\em linkable} if $i \not \sim j$, $g_i g_j \ne 1$ and $\chi_i \chi_j = 1$. The second parameter $\lambda = (\lambda_{i j})_{i, j \in I; i < j}$, called a {\em linking parameter}, is a family of elements of $k$ such that $\lambda_{i j} = 0$ whenever $i$ and $j$ are not linkable.

To describe the third parameter $\mu$, we introduce some notations. Let $\Phi$ be the root system of $A$, $\Phi^+$ a system of positive roots, and $\{ \alpha_i \}_{i \in I}$ a system of simple roots. For a positive root $\alpha = \sum_{i \in I} n_i \alpha_i \in \Phi^+$, we set
\begin{equation*}
  g_\alpha = \prod_{i \in I} g_i^{n_i} \text{\quad and \quad} \chi_\alpha = \prod_{i \in I} \chi_i^{n_i}.
\end{equation*}
Define $J_\alpha \in I/\sim$ so that $n_i \ne 0$ implies $i \in J_\alpha$. We then choose $i \in N_\alpha$ and set $N_\alpha = N_i$. By~\eqref{eq:datum-3}, $N_\alpha$ does not depend on the choice of $i$. The third parameter $\mu = (\mu_\alpha)_{\alpha \in \Phi^+}$, called a {\em root vector parameter}, is a family of elements of $k$ such that $\mu_\alpha = 0$ if either $g_\alpha^{N_\alpha} = 1$ or $\chi_\alpha^{N_\alpha} \ne 1$.

Now we fix a datum $\mathcal{D}$ of finite Cartan type as in~\eqref{eq:datum-0}, a linking parameter $\lambda$ and a root vector parameter $\mu$ for $\mathcal{D}$. Let $V_\mathcal{D}$ be a vector space over $k$ with basis $\{ x_i \}_{i \in I}$. This becomes a left Yetter-Drinfeld module over $k \Gamma$ with the action $\rightharpoonup$ and the coaction $\rho$ given respectively by
\begin{equation*}
  g \rightharpoonup x_i = \chi_i(g) x_i
  \text{\quad and \quad}
  \rho(x_i) = g_i \otimes x_i \quad (g \in \Gamma, i \in I).
\end{equation*}
Consider the braided tensor algebra $B = T(V_{\mathcal{D}})$. For each $\alpha \in \Phi^+$, an element $x_\alpha \in B$ is defined. If $\alpha = \alpha_i$ is a simple root, then $x_\alpha = x_i$. For general $\alpha$, $x_\alpha$ is defined by certain iterated braided commutators. We omit the further detail of the construction of $x_\alpha$'s; see, {\em e.g.}, \cite{MR1886004}. One can find in \cite[\S6]{MR1913436} a concrete description of $x_\alpha$'s for the case where $A$ is of type $A_n$.

Given a Hopf algebra $H$, we set $H^+ = \mathrm{Ker}(\varepsilon_H)$. For each positive root $\alpha \in \Phi^+$, also an element $u_\alpha(\mu) \in k \langle g_i^{N_i} \mid i \in I \rangle^+$ is defined by using the root vector parameter $\mu$. We again omit the detail of the construction, but note that $u_\alpha(\mu) = 0$ for all $\alpha \in \Phi^+$ if $\mu = (0)_{\alpha \in \Phi^+}$; see \cite{MR2630042}.

We now consider the bosonization $\mathcal{U}(\mathcal{D}) := B \# k \Gamma$ and regard both $B$ and $k \Gamma$ as its subalgebras. Let $J(\lambda, \mu)$ be the ideal of $\mathcal{U}(\mathcal{D})$ generated by
\begin{equation*}
  \renewcommand{\arraystretch}{1.25}
  \begin{array}{cc}
    \mathrm{ad}_c(x_i)^{1 - a_{i j}}(x_j)
    & (i, j \in I; i \ne j, i \sim j), \\
    \mathrm{ad}_c(x_i)(x_j) - \lambda_{i j}(1 - g_i g_j)
    & (i, j \in I; i < j, i \not \sim j), \\
    x_\alpha^{N_\alpha} - u_\alpha(\mu)
    & (\alpha \in \Phi^+),
  \end{array}
\end{equation*}
where $\mathrm{ad}_c(x_i) = [x_i, -]_c$ is the braided commutator \eqref{eq:br-commutator}. Now we define
\begin{equation*}
  u(\mathcal{D}, \lambda, \mu) = \mathcal{U}(\mathcal{D}) / J(\lambda, \mu).
\end{equation*}
Abusing notation, the images of $x_i \in B$ and $g \in \Gamma$ in $u(\mathcal{D}, \lambda, \mu)$ under the quotient map are denoted by the same symbols.

Following \cite{MR2630042}, the ideal $J(\lambda, \mu)$ is in fact a Hopf ideal and the quotient $u := u(\mathcal{D}, \lambda, \mu)$ is a finite-dimensional pointed Hopf algebra such that $G(u) = \Gamma$ and
\begin{equation*}
  \dim_k (u) = |\Gamma| \cdot \prod_{J \in I / \mathord{\sim}} N_J^{|\Phi_J^+|},
\end{equation*}
where $\Phi_J^+$ is a system of positive roots for a connected component $J \in I/\mathord{\sim}$. By the construction, the coalgebra structure is determined by
\begin{equation*}
  \Delta(x_i) = x_i \otimes 1 + g_i \otimes x_i \quad (i \in I)
  \text{\quad and \quad} \Delta(g) = g \otimes g \quad (g \in \Gamma).
\end{equation*}

\begin{example}
  \label{ex:small-quantum-group}
  Let $A = (a_{i j})_{i, j = 1, \dotsc, n}$ be a Cartan matrix of finite type with symmetrization $D = \mathrm{diag}(d_1, \dotsc, d_n)$, and let $q$ be a root of unity of order $N > 1$ such that $N$ is odd and $N \not \equiv 0 \pmod{3}$ if the Dynkin diagram of $A$ has a connected component of type $G_2$. Let $\Gamma$ be the abelian group generated by $g_1, \dotsc, g_n$ with defining relations $g_i^N = 1$ ($i = 1, \dotsc, n$) and set $I = \{ 1, 2, \dotsc, 2 n \}$,
  \begin{equation*}
    g_{i + n} = g_i, \quad
    \chi_i(g_j) = q^{d_i a_{i j}},
    \quad \chi_{i + n} = \chi_i^{-1}
    \text{\quad and \quad}
    \widetilde{A} = \begin{pmatrix} A & 0 \\ 0 & A \end{pmatrix}
  \end{equation*}
  for $i, j = 1, \dotsc, n$. It is easy to check that $\mathcal{D} = (\Gamma, (g_i)_{i \in I}, (\chi_i)_{i \in I}, \widetilde{A})$ is a datum of finite Cartan type for $\Gamma$.

  (1) Define $\lambda_{i j} = -\delta_{i+n,j}(q^{d_i} - q^{-d_i})^{-1}$ for $i, j \in I$ with $i < j$. Then $\lambda = (\lambda_{i j})$ is a linking parameter for $\mathcal{D}$. The Hopf algebra $u_q(\mathfrak{g}) := u(\mathcal{D}, \lambda, 0)$ is known as the small quantum group associated with the semisimple Lie algebra $\mathfrak{g}$ corresponding to $A$. The usual generators of $u_q(\mathfrak{g})$ are given by $E_i = x_i$, $K_i = g_i$ and $F_i = x_{i+n} g_i^{-1}$ ($i = 1, \dotsc, n$).

  (2) It is easy to see that $\mathcal{E} = (\Gamma, (g_i)_{i = 1}^n, (\chi_i)_{i = 1}^n, A)$ is a datum of finite Cartan type for $\Gamma$. The Hopf algebra $u(\mathcal{E}, 0, 0)$ is isomorphic to the Hopf subalgebra $u_q^{\ge 0}(\mathfrak{g})$ of $u_q(\mathfrak{g})$ generated by $E_i, K_i$ ($i = 1, \dotsc, n$).

  (3) $\mathcal{F} = (\Gamma, (g_{i + n})_{i = 1}^n, (\chi_{i + n})_{i = 1}^n, A)$ is also a datum of finite Cartan type for $\Gamma$. The Hopf algebra $u(\mathcal{F}, 0, 0)$ is isomorphic to the Hopf subalgebra $u_q^{\le 0}(\mathfrak{g})$ of $u_q(\mathfrak{g})$ generated by $F_i, K_i$ ($i = 1, \dotsc, n$). Note that $V_{\mathcal{F}}$ is isomorphic to $V_{\mathcal{E}}^{\op}$ as a Yetter-Drinfeld module over $k \Gamma$ (see \S\ref{subsec:bosonization} for the definition of $X^{\op}$ for $X \in {}^\Gamma_\Gamma \YD$). Hence, by Lemma~\ref{lem:boson-op},
  \begin{equation}
    \label{eq:uq-borel-op}
    u_q^{\ge 0}(\mathfrak{g})^{\op} \cong \mathfrak{B}(V_{\mathcal{E}}^{\op}) \# k \Gamma
    \cong \mathfrak{B}(V_{\mathcal{F}}) \# k \Gamma \cong u_q^{\le 0}(\mathfrak{g}).
  \end{equation}
\end{example}

We go back to the general situation. If we consider the coradical filtration, then there are isomorphisms $\gr u(\mathcal{D}, \lambda, \mu) \cong u(\mathcal{D}, 0, 0) \cong \mathfrak{B}(V_\mathcal{D}) \# k \Gamma$ of (graded) Hopf algebras. Hence, by Theorem~\ref{thm:ind-filtered}, we have:

\begin{theorem}
  \label{thm:ind-pointed-u}
  $\nu_n(u(\mathcal{D}, \lambda, \mu)) = \nu_n(\mathfrak{B}(V_\mathcal{D}) \# k \Gamma)$ for all $n \ge 1$.
\end{theorem}

\subsection{Factorization of $\nu_2$}

Fix a finite abelian group $\Gamma$ of odd order. Let $\mathcal{D}$ be a datum of finite Cartan type for $\Gamma$ as in \eqref{eq:datum-0}. For each $J \in I/\sim$,
\begin{equation*}
  \mathcal{D}[J] := (\Gamma, (g_i)_{i \in J}, (\chi_i)_{i \in J}, (a_{i j})_{i, j \in J})
\end{equation*}
is also a datum of finite Cartan type for $\Gamma$ ({\it cf.} Example~\ref{ex:small-quantum-group} (2) and (3)). The main result of this section is the following {\em factorization formula}:

\begin{theorem}
  \label{thm:2nd-ind-factor}
  For any parameters $\lambda$ and $\mu$ for the above $\mathcal{D}$,
  \begin{equation*}
    \nu_2(u(\mathcal{D}, \lambda, \mu)) = \prod_{J \in I/\sim} \nu_2(\mathfrak{B}(V_{\mathcal{D}[J]}) \# k \Gamma).
  \end{equation*}
\end{theorem}

To prove this theorem, we introduce a natural isomorphism
\begin{equation*}
  \theta_V: V \to V, \quad \theta_V(v) = v_{(-1)} \rightharpoonup v_{(0)} \quad (v \in V \in {}^\Gamma_\Gamma \YD).
\end{equation*}
Let $V \in {}^\Gamma_\Gamma \YD$. Recall that $V$ is $\Gamma$-graded in a natural way. For all homogeneous element $v \in V_g$ and for all $m \in \mathbb{Z}$, we have $\theta_V^m(v) = g^m \rightharpoonup v$. In particular, $\theta^{|\Gamma|} = \id$. Recalling our assumption that $|\Gamma|$ is odd, we introduce symbols
\begin{equation*}
  \theta^{1/2} = \theta^{(|\Gamma|+1)/2} \text{\quad and \quad}
  \theta^{-1/2} = (\theta^{1/2})^{-1} = \theta^{-(|\Gamma|+1)/2}.
\end{equation*}
Now we suppose that $B$ is a finite-dimensional braided Hopf algebra over $k \Gamma$. The second indicator of $B \# k\Gamma$ can be expressed by $\theta$ and $S_B$ as follows:

\begin{lemma}
  \label{lem:2nd-ind-lem-1}
  $\nu_2(B \# k \Gamma) = \Trace(\theta_B^{-1/2} \circ S_B)$.
\end{lemma}
\begin{proof}
  Fix $g, h \in \Gamma$ and a homogeneous element $b \in B_g$. By~(\ref{eq:boson-antipode}),
  \begin{equation}
    \label{eq:boson-antipode-1}
    S_{B \# k\Gamma}(b \# h)
    = (h^{-1} g^{-1} \rightharpoonup S_B(b)) \# (h^{-1} g^{-1}).
  \end{equation}
  Since $S_B$ preserves the coaction of $k\Gamma$, $S_B(b) \in B_g$. Therefore we have
  \begin{equation}
    \label{eq:boson-antipode-2}
    S_{B \# k\Gamma}(B_g \# h) \subset B_g \# (h^{-1} g^{-1})
    \quad (g, h \in \Gamma).
  \end{equation}
  Following~\eqref{eq:boson-antipode-2}, for each $g, h \in \Gamma$, we define $S_{g,h}: B_g \# h \to B_g \# (h^{-1} g^{-1})$ to be the restriction of $S_{B \# k \Gamma}$. Since $B \# k \Gamma = \bigoplus_{g,h \in \Gamma} B_g \# h$, we have
  \begin{equation}
    \label{eq:boson-antipode-3}
    \nu_2(B \# k \Gamma) = \Trace(S_{B \# k \Gamma}) = \sum \Trace(S_{g,h}),
  \end{equation}
  where the sum is taken over all $(g, h) \in \Gamma \times \Gamma$ such that $h = h^{-1} g^{-1}$. Recalling our assumption that $|\Gamma|$ is odd, we set $g^{-1/2} := g^{-(|\Gamma|+1)/2}$ for $g \in \Gamma$. Then $h = h^{-1} g^{-1}$ if and only if $h = g^{-1/2}$. By~\eqref{eq:boson-antipode-1}, we have
  \begin{equation*}
    S_{g, g^{-1/2}}(b \# g^{-1/2})
    = (g^{-1/2} \rightharpoonup S_B(b)) \# g^{-1/2}
    = \theta_B^{-1/2} S_B(b) \# g^{-1/2}
  \end{equation*}
  for all $b \in B_g$, and hence $\Trace(S_{g,g^{-1/2}}) = \Trace(\theta_B^{-1/2} S_B|_{B_g}: B_g \to B_g)$. By \eqref{eq:boson-antipode-3}, we now conclude $\nu_2(B \# k \Gamma) = \sum_{g \in \Gamma} \Trace(S_{g,g^{-1/2}}) = \Trace(\theta_B^{-1/2} S_B)$.
\end{proof}

For $X, Y \in {}^\Gamma_\Gamma \YD$, one has $\theta_{X \otimes Y} = (\theta_X \otimes \theta_Y) c_{Y,X} c_{X,Y}$. Hence $X$ and $Y$ centralize each other if and only if
\begin{equation}
  \label{eq:YD-twist-commute}
  \theta_{X \otimes Y} = \theta_X \otimes \theta_Y.
\end{equation}

\begin{lemma}
  \label{lem:2nd-ind-lem-2}
  Let $B_1$ and $B_2$ be finite-dimensional braided Hopf algebras over $k \Gamma$. If they centralize each other, then we have
  \begin{equation*}
    \nu_2((B_1 \mathop{\underline{\otimes}} B_2) \# k \Gamma)
    = \nu_2(B_1 \# k \Gamma) \cdot \nu_2(B_2 \# k \Gamma).
  \end{equation*}
\end{lemma}
\begin{proof}
  We get $\theta_{B_1 \mathop{\underline{\otimes}} B_2}^{-1/2} \circ S_{B_1 \mathop{\underline{\otimes}} B_2}^{} = (\theta_{B_1}^{-1/2} \circ S_{B_1}) \otimes (\theta_{B_2}^{-1/2} \circ S_{B_2})$ from \eqref{eq:YD-twist-commute} and the definition of the braided tensor product. Take the trace of both sides and then apply Lemma~\ref{lem:2nd-ind-lem-1}.
\end{proof}

\begin{proof}[Proof of Theorem~\ref{thm:2nd-ind-factor}]
  Write $I/\mathord{\sim} = \{ I_1, \dotsc, I_m \}$ and set $\mathfrak{B}_j := \mathfrak{B}(V_{\mathcal{D}[I_j]})$ for each $j$. Then, for each $i, j$ with $i \ne j$, $\mathfrak{B}_i$ and $\mathfrak{B}_j$ centralize each other, and there is an isomorphism $\mathfrak{B}(V_{\mathcal{D}}) \cong \mathfrak{B}_1 \mathop{\underline{\otimes}} \dotsb \mathop{\underline{\otimes}} \mathfrak{B}_m$ of braided Hopf algebras \cite[Lemma~1.4]{MR2108213}. Now the result is obtained by using Lemma~\ref{lem:2nd-ind-lem-2} repeatedly.
\end{proof}

\begin{corollary}
  \label{cor:2nd-ind-factor-uq}
  Let $\mathfrak{g}$ and $q$ be as in Example~\ref{ex:small-quantum-group}, and decompose $\mathfrak{g}$ into a direct sum of simple Lie algebras, as $\mathfrak{g} = \mathfrak{g}_1 \oplus \dotsb \oplus \mathfrak{g}_m$. Then we have
  \begin{equation*}
    \nu_2(u_q(\mathfrak{g}))
    = \prod_{i = 1}^m \Big| \nu_2(u_q^{\ge 0}(\mathfrak{g}_i)) \Big|^2.
  \end{equation*}
\end{corollary}
\begin{proof}
  By Theorem~\ref{thm:2nd-ind-factor}, $\nu_2(u_q(\mathfrak{g})) = \prod_{i = 1}^m \nu_2(u_q^{\ge 0}(\mathfrak{g}_i)) \cdot \nu_2(u_q^{\le 0}(\mathfrak{g}_i))$ Hence the desired formula follows from~\eqref{eq:uq-borel-op} and Proposition \ref{prop:ind-cyc-int} (2).
\end{proof}

\subsection{Examples and an application to $u_q(\mathfrak{sl}_2)$}
\label{sec:ex-pointed}

As an application of Corollary \ref{cor:2nd-ind-factor-uq}, we compute the second indicator of $u_q(\mathfrak{sl}_2)$. For this purpose, we need the second indicator of $u_q^{\ge 0}(\mathfrak{sl}_2)$, {\em i.e.}, the Taft algebra.

\begin{example}[Taft algebras]
  \label{ex:ind-2nd-odd-Taft}
  Let $\omega$ be a root of unity of odd order $N > 1$, and let $\Gamma$ be a cyclic group of order $N$ generated by $g \in \Gamma$. Define $V \in {}^\Gamma_\Gamma \YD$ by
  \begin{equation*}
    V = k x,
    \quad g \rightharpoonup x = \omega x,
    \quad \rho(x) = g \otimes x.
  \end{equation*}
  The Nichols algebra $B := \mathfrak{B}(V)$ is generated as an algebra by the primitive element $x \in V$ with relation $x^N = 0$ and its bosonization $B \# k \Gamma$ is isomorphic to the Taft algebra $T_{N^2}(\omega)$.

  Let us compute $\nu_2(T_{N^2}(\omega))$ by using Lemma~\ref{lem:2nd-ind-lem-1}. Since $x \in B_g$, we have $x^i \in B_{g^i}$ for $i = 0, \dotsc, N-1$. Hence $\theta_B(x^i) = g^i \rightharpoonup x^i = \omega^{i^2} x^i$. By using \eqref{eq:br-Hopf-antipode} repeatedly, we also have
  \begin{equation*}
    S_B(x^i) = S_B(g \rightharpoonup x^{i - 1}) \cdot S_B(x)
    = - \omega^i S_B(x^{i-1}) x
    = \dotsc = (-1)^i \omega^{i(i-1)/2} x^i.
  \end{equation*}
  For simplicity, set $\xi = \omega^{(N+1)/2}$ so that $\xi^2 = \omega$. By the above computation,
  \begin{equation*}
    (\theta_B^{-1/2} \circ S_B)(x^i)
    = (-1)^i \xi^{i(i-1) - i^2} x^i = (-\xi^{-1})^{i} x^i.
  \end{equation*}
  By Lemma~\ref{ex:ind-2nd-odd-Taft},
  \begin{equation*}
    \nu_2(T_{N^2}(\omega)) = \sum_{i = 0}^{N - 1} (-\xi^{-1})^{i}
    = \frac{1 - (-\xi^{-1})^N}{1 - (-\xi^{-1})}
    = \frac{2}{1 + \omega^{-(N+1)/2}}.
  \end{equation*}
  This formula has been obtained in \cite{KMN09} in a more direct approach (but notice that our $T_{N^2}(\omega)$ is isomorphic to their $T_{N^2}(\omega^{-1})$). Note that our arguments cannot be applied if $N$ is even. In \cite{KMN09} they have also showed that if $N = \ord(\omega)$ is even, then $\nu_2(T_{N^2}(\omega)) = 4 (1 - \omega^{-1})^{-1}$.
\end{example}

\begin{example}
  \label{ex:ind-2nd-uqsl2}
  Let $q$ be a root of unity of odd order $N > 1$. The Hopf algebra $u_q(\mathfrak{sl}_2)$ is generated as an algebra by $E$, $F$ and $K$ with relations $E^N = F^N = 0$, $K^N = 1$, $K E = q^2 E K$, $K F = q^{-2} F K$, and
  \begin{equation}
    \label{eq:uqsl2-linking}
    E F - F E = \frac{K - K^{N - 1}}{q - q^{-1}}.
  \end{equation}
  The coalgebra structure is determined by
  \begin{equation*}
    \Delta(E) = E \otimes 1 + K \otimes E,
    \ \Delta(F) = F \otimes K^{-1} + 1 \otimes F,
    \ \Delta(K) = K \otimes K.
  \end{equation*}
  Now we apply Corollary~\ref{cor:2nd-ind-factor-uq}. Since $u_q^{\ge 0}(\mathfrak{sl}_2) \cong T_{N^2}(q^2)$,
  \begin{equation}
    \label{eq:ind-2nd-uqsl2}
    \nu_2(u_q(\mathfrak{sl}_2))
    = \nu_2(T_{N^2}(q^2)) \cdot \overline{\nu_2(T_{N^2}(q^2))}
    = \frac{4}{(1 + q)(1 + q^{-1})}.
  \end{equation}
  Let $p$ be another root of unity with $\ord(p) = \ord(q)$. As an application of~\eqref{eq:ind-2nd-uqsl2}, we can prove the following theorem:

  \begin{theorem}
    The following three assertions are equivalent:
    \begin{itemize}
    \item [(1)] $u_p(\mathfrak{sl}_2)$ and $u_q(\mathfrak{sl}_2)$ are isomorphic as Hopf algebras.
    \item [(2)] $u_p(\mathfrak{sl}_2)$ and $u_q(\mathfrak{sl}_2)$ are gauge equivalent.
    \item [(3)] $p = q^{\pm 1}$.
    \end{itemize}
  \end{theorem}
  \begin{proof}
    (1) $\Rightarrow$ (2) is trivial. To prove (2) $\Rightarrow$ (3), we compute
    \begin{equation*}
      \frac{1}{\nu_2(u_p(\mathfrak{sl}_2))} - \frac{1}{\nu_2(u_q(\mathfrak{sl}_2))}
      = \frac{(p - q)(1 - p^{-1} q^{-1})}{4}
    \end{equation*}
    by using~\eqref{eq:ind-2nd-uqsl2}. By the gauge invariance of $\nu_2$, we obtain $p = q^{\pm 1}$. Finally, we prove (3) $\Rightarrow$ (1). If $p = q$, then the claim is trivial. If $p = q^{-1}$, then
    \begin{equation*}
      u_q(\mathfrak{sl}_2) \to u_p(\mathfrak{sl}_2);
      \quad E \mapsto F K,
      \quad F \mapsto K^{-1} E,
      \quad K \mapsto K
    \end{equation*}
    defines an isomorphism of Hopf algebras \cite[\S 3.1.2]{MR1492989}.
  \end{proof}
\end{example}

\section{Higher indicators of Taft algebras}
\label{sec:higher-ind-taft}

\subsection{Computing the Sweedler power maps}

Throughout, $k$ is assumed to be an algebraically closed field of characteristic zero. In this section, we derive closed formulas for the indicators of the Taft algebra $T_{N^2}(\omega)$ and that of the small quantum group $u_q(\mathfrak{sl}_2)$. As a result, it turns out that the $n$-th indicator of $u_q(\mathfrak{sl}_2)$ is expressed by the $n$-th indicator of $T_{N^2}(\omega)$ with $\omega = q^2$.

We first introduce several notations. Let $q$ be a formal parameter. For an integer $n \ge 0$, we set $(0)_q = 0$ and $(n)_q = 1 + q + \dotsb + q^{n - 1}$ ($n \ge 1$). The $q$-factorial $(n)_q!$ is defined inductively by $(0)_q! = 1$ and $(n)_q! = (n - 1)_q! \cdot (n)_q$. For integers $m$ and $a$, we define the $q$-binomial coefficient by
\begin{equation}
  \label{eq:q-binom}
  \binom{m}{a}_{\!\! q} = \frac{(m)_q!}{(a)_q! (m - a)_q!}
  \quad (0 \le a \le m) \text{\quad and \quad}
  \binom{m}{a}_{\!\! q} = 0
  \quad (\text{otherwise}).
\end{equation}
For non-negative integers $L$, $a$ and $m$, we set
\begin{equation}
  \label{eq:q-function}
  \left\{ {L \atop a, m } \right\}_{\!q} = \sum_{j_1 + \dotsb + j_m = a} q^{j_1^2 + \dotsb + j_m^2}
  \binom{L}{j_1}_{\!\!q} \binom{j_1}{j_2}_{\!\!q} \dotsm \binom{j_{m - 1}}{j_m}_{\!\!q},
\end{equation}
where the sum is taken over all non-negative integers $j_1, \dotsc, j_m$ satisfying $j_1 + \dotsb + j_m = a$. Note that the summand of the right-hand side of~\eqref{eq:q-function} is zero unless
\begin{equation}
  \label{eq:partition}
  L \ge j_1 \ge j_2 \ge \dotsc \ge j_m \ge 0
\end{equation}
is satisfied. Thus \eqref{eq:q-function} is a sum taken over all partitions of $a$. Note also
\begin{equation}
  \label{eq:q-function-zero}
  \left\{ {L \atop a, m } \right\}_{\!q} = 0
  \text{\quad unless \quad} 0 \le a \le m L.
\end{equation}

\begin{remark}
  \label{rem:q-function}
  The right-hand side of \eqref{eq:q-function} has a combinatorial interpretation; indeed, it is the generating function of $(m, m + 1; L, a)$-admissible partitions in the sense of Warnaar \cite[Definition~7]{MR1462505}.
\end{remark}

Now fix a parameter $\omega \in k^\times$. Since \eqref{eq:q-binom} and \eqref{eq:q-function} are Laurent polynomials of $q$, we can substitute $q = \omega$. Let $T(\omega)$ be the algebra generated by $g$, $g^{-1}$ and $x$ with defining relations $g g^{-1} = 1 = g^{-1} g$ and $g x g^{-1} = \omega x$. $T(\omega)$ admits a Hopf algebra structure determined by the same formula as~(\ref{eq:taft-coalg}) and has $T_{N^2}(\omega)$ as a quotient Hopf algebra when $\omega$ is a primitive $N$-th root of unity.

\begin{lemma}
  \label{lem:Sw-pow-Taft}
  The $n$-th Sweedler power of $x^r g^s \in T(\omega)$ is given by
  \begin{equation*}
    (x^r g^s)^{[n]} = \sum_{a = 0}^{\infty} \left\{ {r \atop a, n - 1} \right\}_{\!\omega}
    \omega^{a s} x^r g^{a + n s}.
  \end{equation*}
\end{lemma}

By~\eqref{eq:q-function-zero}, the right-hand side is in fact a sum taken over $0 \le a \le r(n - 1)$. For simplicity of notation, we are expressing it as an infinite sum.

\begin{proof}
  Define $\varphi: T(\omega) \to T(\omega)$ by $u \mapsto g u g^{-1}$. By Lemma~\ref{lem:Sw-pow-map}, we have
  \begin{equation}
    \label{eq:Sw-power-1}
    (x^r g^s)^{[n]} = (\nabla^{(n)} \circ
    (\id \otimes \varphi^s \otimes \dotsb \otimes \varphi^{s(n - 1)})
    \circ \Delta^{(n)} ) (x^r) \cdot g^{n s},
  \end{equation}
  where $\Delta^{(n)}$ and $\nabla^{(n)}$ are given by~\eqref{eq:iterated-mult} and~\eqref{eq:iterated-comul}, respectively. In what follows, we use multi-index notation; an {\em $n$-dimensional multi-index} is an $n$-tuple $\mathbf{i} = (i_1, \dotsc, i_n)$ of non-negative integers. Given such $\mathbf{i}$, we set
  \begin{equation*}
    |\mathbf{i}| = i_1 + \dotsb + i_n
    \text{\quad and \quad}
    \binom{m}{\mathbf{i}}_{\!\!q} = \frac{(m)_q!}{(i_1)_q! \dotsm (i_n)_q!}
    \quad (\text{for $|\mathbf{i}| = m$}).
  \end{equation*}
  Observe that $\Delta^{(n)}(x) = x_1 + \dotsb + x_n$, where
  \begin{equation}
    \label{eq:Sw-power-xi}
    x_i =
    \underbrace{g \otimes \dotsb \otimes g}_{i - 1} \otimes x \otimes
    \underbrace{1 \otimes \dotsb \otimes 1}_{n - i} \in T(\omega)^{\otimes n}
    \quad (i = 1, \dotsc, n).
  \end{equation}
  Let $I_{n,r}$ be the set of all $n$-dimensional multi-indices $\mathbf{i}$ such that $|\mathbf{i}| = r$. Note that if $i < j$, then $x_j x_i = \omega x_i x_j$. Hence, by the $q$-multinomial formula,
  \begin{equation*}a
    \Delta^{(n)}(x^r)
    = \Delta^{(n)}(x)^r
    = (x_1 + \dotsb + x_n)^r
    = \sum_{\mathbf{i} \in I_{n, r}} \binom{r}{\mathbf{i}}_{\!\!q} \mathbf{x}^{\mathbf{i}}
  \end{equation*}
  where $\mathbf{x}^{\mathbf{i}} = x_1^{i_1} \dotsm x_{n}^{i_n}$. By a slightly tedious computation, we obtain
  \begin{equation*}
    (\id \otimes \varphi \otimes \dotsm \otimes \varphi^{n - 1}) (\mathbf{x}^{\mathbf{i}})
    = \omega^{\sigma_1(\mathbf{i})} \mathbf{x}^{\mathbf{i}} 
    \text{\quad and \quad}
    \nabla^{(n)} (\mathbf{x}^{\mathbf{i}}) = \omega^{\sigma_2(\mathbf{i})} x^r g^{\sigma_1(\mathbf{i})},
  \end{equation*}
  where $\sigma_m(\mathbf{i}) = \sum_{c = 1}^n \, (i_{c + 1} + \dotsm + i_{n})^m$ for $m \ge 1$. Hence, by~(\ref{eq:Sw-power-1}),
  \begin{equation}
    \label{eq:Sw-power-2}
    (x^r g^s)^{[n]} = \sum_{a = 0}^{\infty} \, \sum_{\smash{\sigma_1(\mathbf{i})} = a}
    \binom{r}{\mathbf{i}}_{\!\omega} \omega^{a s + \sigma_2(\mathbf{i})} x^r g^{a + n s},
  \end{equation}
  where the second sum is taken over all $\mathbf{i} \in I_{n,r}$ satisfying $\sigma_1(\mathbf{i}) = a$.

  Now let $J_{n, r}$ be the set of all $(n - 1)$-tuple $(j_1, \dotsc, j_{n - 1})$ of non-negative integers satisfying~(\ref{eq:partition}) with $L = r$ and $m = n - 1$. Then the map
  \begin{equation*}
    f: I_{n, r} \to J_{n, r},
    \ (i_1, \dotsc, i_n) \mapsto
    (i_2 + \dotsb + i_n, i_3 + \dotsb + i_n, \dotsc, i_{n - 1} + i_n, i_n)
  \end{equation*}
  gives a bijection between $I_{n,r}$ and $J_{n,r}$. If $(j_1, \dotsc, j_{n - 1}) = f(\mathbf{i})$, then
  \begin{equation*}
    \sigma_m(\mathbf{i}) = j_1^m + j_2^m + \dotsb + j_{n - 1}^m \text{\quad and \quad}
    \binom{r}{\mathbf{i}} = \binom{r}{j_1}_{\!\!q} \binom{j_1}{j_2}_{\!\!q} \! \dotsm \binom{j_{n-2}}{j_{n-1}}_{\!\!q}.
  \end{equation*}
  The desired formula is obtained by rewriting (\ref{eq:Sw-power-2}) as the sum taken over all $a \ge 0$ and all $(j_1, \dotsc, j_{n - 1}) \in J_{n, r}$ such that $j_1 + \dotsb + j_{n - 1} = a$.
\end{proof}

\subsection{Higher indicators of the Taft algebra}
\label{subsec:indicator-Taft}

Let $N > 1$ be an integer, and let $\omega \in k$ be a primitive $N$-th root of unity. Here we compute the indicators of the Taft algebra $T_{N^2}(\omega)$. The following lemma will be needed:

\begin{lemma}
  \label{lem:cong-eq}
  Consider the congruence equation $n x \equiv m \pmod{N}$, where $n$ and $m$ are integers and $N$ is a positive integer.
  \begin{enumerate}
  \item If $m \not \equiv 0 \pmod{d}$, where $d = \gcd(N, n)$, then there are no solutions.
  \item If $m \equiv 0 \pmod{d}$, then the solutions are
    $x = \check{n} (m/d) + (N/d) j + N \mathbb{Z} \ (j = 0, 1, \dotsc, d - 1)$,
    where $\check{n}$ is an integer such that $\check{n} n \equiv d \pmod{N}$.
  \end{enumerate}
\end{lemma}

The proof is elementary and omitted.

\begin{theorem}
  \label{thm:higher-ind-Taft}
  Let $\omega$ be a root of unity of order $N > 1$. Then we have
  \begin{equation}
    \label{eq:higher-ind-Taft-1}
    \nu_n(T_{N^2}(\omega)) = d \sum_{i = 0}^\infty
    \left\{ N - 1 \atop d i, n - 1 \right\}_\omega \, \omega^{-\check{n} d i^2}
  \end{equation}
  for all $n \ge 1$, where $d = \gcd(N, n)$ and $\check{n}$ is an integer such that $\check{n} n \equiv d \pmod{N}$.
\end{theorem}
\begin{proof}
  Put $\Lambda = \sum_{s = 0}^{N - 1} x^{N - 1} g^{-s} \in T_{N^2}(\omega)$ and define $\lambda \in T_{N^2}(\omega)^*$ by
  \begin{equation*}
    \langle \lambda, x^i g^j \rangle = \delta_{i, N - 1} \delta_{j, 0}
    \quad (i, j = 0, \dotsc, N - 1).
  \end{equation*}
  $\Lambda$ and $\lambda$ are right integrals such that $\langle \lambda, \Lambda \rangle = 1$. By Lemmas~\ref{lem:ind-integral} and~\ref{lem:Sw-pow-Taft},
  \begin{equation}
    \label{eq:higher-ind-Taft-2}
    \nu_n(T) = \sum_{a = 0}^\infty \sum_{\smash{s \in K(a)}}
    \left\{ N - 1 \atop \ell, n - 1 \right\}_\omega \omega^{- a s},
  \end{equation}
  where $K(a) = \{ x \in \mathbb{Z}/N\mathbb{Z} \mid n x \equiv a \pmod{N} \}$. By Lemma~\ref{lem:cong-eq}, $K(a) = \emptyset$ unless $a \equiv 0 \pmod{d}$. Hence, we may assume that the first sum of~\eqref{eq:higher-ind-Taft-2} is taken over all $a \in d \mathbb{Z}_{\ge 0}$. If $a = d i$ for some $i \in \mathbb{Z}_{\ge 0}$, then, by Lemma~\ref{lem:cong-eq},
  \begin{equation*}
    K(a) = \{ \check{n} i + (N/d) j + N \mathbb{Z} \mid j = 0, \dotsc, d - 1 \}.
  \end{equation*}
  Observe $d s \equiv \check{n} d i^2 \pmod{N}$ for all $s \in K(a)$. By~\eqref{eq:higher-ind-Taft-2}, we obtain
  \begin{equation*}
    \nu_n(T)
    = \sum_{i = 0}^\infty \sum_{\smash{s \in K(d i)}}
    \left\{ N - 1 \atop d i, n - 1 \right\}_\omega \omega^{- d i s}
    = d \sum_{i = 0}^\infty
    \left\{ N - 1 \atop d i, n - 1 \right\}_\omega \omega^{-\check{n} d i^2}. \qedhere
  \end{equation*}
\end{proof}

\subsection{Higher indicators of $u_q(\mathfrak{sl}_2)$}
\label{subsec:indicator-uqsl2}

We now show that the indicators of $u_q(\mathfrak{sl}_2)$ can be expressed by the indicators of the Taft algebra, as follows:

\begin{theorem}
  \label{thm:higher-ind-uqsl2}
  Let $q$ be a root of unity of odd order $N > 1$. Then we have
  \begin{equation*}
    \nu_n(u_q(\mathfrak{sl}_2)) = \frac{1}{\gcd(N, n)} \, \big| \nu_n(T_{N^2}(q^2)) \big|^2
  \end{equation*}
  for all $n \ge 1$.
\end{theorem}

To prove this theorem, we consider the coradical filtration of $u_q := u_q(\mathfrak{sl}_2)$ and let $u_q' = \gr(u_q)$. By Theorem~\ref{thm:ind-filtered}, we have $\nu_n(u_q) = \nu_n(u_q')$ for all $n \ge 1$. Thus we compute $\nu_n(u_q')$ instead of $\nu_n(u_q)$. As an algebra, $u_q'$ is generated by $K$, $E$ and $F$ with the same relations as $u_q$ but with \eqref{eq:uqsl2-linking} replaced by $E F - F E = 0$. Now we define elements $g, x, y \in u_q'$ by $g = K$, $x = E$ and $y = F K$. It is easy to see that $u_q'$ is generated as an algebra by $g$, $x$ and $y$ with relations
\begin{equation*}
  x^N = y^N = 0, \quad g^N = 1,
  \quad g x = q^2 x g,
  \quad g y = q^{-2} y g,
  \quad y x = q^2 x y.
\end{equation*}
With respect to these generators, the comultiplication is given by
\begin{equation*}
  \Delta(g) = g \otimes g,
  \quad \Delta(x) = x \otimes 1 + g \otimes x,
  \quad \Delta(y) = y \otimes 1 + g \otimes y.
\end{equation*}
Note that $u_q'$ has the set $\{ x^r y^s g^\ell \mid r, s, \ell = 0, 1, \dotsc, N - 1 \}$ as a basis. We first compute the Sweedler power maps with respect to this basis:

\begin{lemma}
  \label{lem:Sw-pow-gr-uqsl2}
  Let $n$, $r$, $s$ and $t$ be integers such that $n > 0$ and $r, s \ge 0$. Then the $n$-th Sweedler power of $x^r y^s g^t \in u_q'$ is given by
  \begin{equation*}
    (x^r y^s g^t)^{[n]} = \sum_{a, b = 0}^{\infty}
    \left\{ {r \atop a, n - 1} \right\}_{q^2}
    \left\{ {s \atop b, n - 1} \right\}_{q^{-2}}
    q^{2 t (a - b)} x^r y^s g^{a + b + n t}.
  \end{equation*}
\end{lemma}
\begin{proof}
  Define $\varphi: u_q' \to u_q'$ by $u \mapsto g u g^{-1}$. By Lemma~\ref{lem:Sw-pow-map}, we have
  \begin{equation}
    \label{eq:Sw-pow-map-3}
    (x^r y^s g^\ell)^{[n]} = (\nabla^{(n)}
    \circ (\id \otimes \varphi^t \otimes \dotsb \otimes \varphi^{t(n - 1)})^\ell
    \circ \Delta^{(n)}) (x^r y^s) \cdot g^{n t}.
  \end{equation}
  For a multi-index $\mathbf{i} = (i_1, \dotsc, i_n)$, we set $\mathbf{x}^{\mathbf{i}} = x_1^{i_1} \dotsb x_n^{i_n}$ and $\mathbf{y}^{\mathbf{i}} = y_1^{i_1} \dotsb y_n^{i_n}$, where $x_i$ is given by~(\ref{eq:Sw-power-xi}) and $y_i$ is given by the same formula as $x_i$ but with $x$ replaced by $y$. Now let $I_{n, r}$ be the set of multi-indices $\mathbf{i}$ such that $|\mathbf{i}| = r$. Then, by the $q$-multinomial formula, we have
  \begin{equation*}
    \Delta^{(n)}(x^r y^s)
    = \Delta^{(n)}(x)^r \cdot \Delta^{(n)}(y)^s
    = \sum_{\mathbf{i} \in I_{n, r}} \sum_{\mathbf{j} \in I_{n, s}}
    \binom{r}{\mathbf{i}}_{q^2} \binom{r}{\mathbf{j}}_{q^{-2}}
    \mathbf{x}^{\mathbf{i}} \mathbf{y}^{\mathbf{j}}.
  \end{equation*}
  For $\mathbf{i} \in I_{n,r}$, set $\sigma_m(\mathbf{i}) = \sum_{c = 1}^n \, (i_{c + 1} + \dotsm + i_{n})^m$. Then we have
  \begin{equation*}
    (\id \otimes \varphi \otimes \dotsb \otimes \varphi^{n - 1})(\mathbf{x}^{\mathbf{i}} \mathbf{y}^{\mathbf{j}})
    =q^{2(\sigma_1(\mathbf{i}) - \sigma_1(\mathbf{j}))} \mathbf{x}^{\mathbf{i}} \mathbf{y}^{\mathbf{j}}.
  \end{equation*}
  Writing $X_1 \dotsb X_n$ as $\prod_{c = 1}^n X_c$ for short, we compute
  \begin{align*}
    \nabla^{(n)}(\mathbf{x}^{\mathbf{i}} \mathbf{y}^{\mathbf{j}})
    & = \prod_{c = 1}^n x^{i_c} g^{i_{c+1} + \dotsb + i_n} y^{j_c} g^{j_{c+1} + \dotsb + j_n}
    = q^{2(\sigma_2(\mathbf{i}) - \sigma_2(\mathbf{j}))} x^{|\mathbf{i}|} y^{|\mathbf{j}|} g^{\sigma_1(\mathbf{i}) + \sigma_1(\mathbf{j})}.
  \end{align*}
  Put $\tilde{\sigma}(\mathbf{i}) = t \cdot \sigma_1(\mathbf{i}) + \sigma_2(\mathbf{i})$ for simplicity. Now we have
  \begin{equation*}
    (x^r y^s g^t)^{[n]} = \sum_{\ \mathbf{i} \in I_{n, r}} \sum_{\mathbf{j} \in I_{n, s}}
    \binom{r}{\mathbf{i}}_{q^2} \binom{r}{\mathbf{j}}_{q^{-2}}
    q^{2 \tilde{\sigma}(\mathbf{i}) - 2 \tilde{\sigma}(\mathbf{j})}
    x^r y^s g^{\sigma_1(\mathbf{i}) + \sigma_1(\mathbf{j}) + n t}
  \end{equation*}
  by \eqref{eq:Sw-pow-map-3}. The desired formula is obtained by rewriting the right-hand side in a similar way as the proof of Lemma~\ref{lem:Sw-pow-Taft}.
\end{proof}

\begin{proof}[Proof of Theorem~\ref{thm:higher-ind-uqsl2}]
  Put $\Lambda = \sum_{\ell = 0}^{N - 1} x^{N - 1} y^{N - 1} g^{-t} \in u_q'$ and define $\lambda \in (u_q')^*$ by
  $\langle \lambda, x^r y^s g^t \rangle = \delta_{r, N - 1} \delta_{s, N - 1} \delta_{t, 0}$
  ($r, s, t = 0, \dotsc, N - 1$). $\lambda$ and $\Lambda$ are both right integrals such that $\langle \lambda, \Lambda \rangle = 1$. By Lemmas~\ref{lem:ind-integral} and~\ref{lem:Sw-pow-gr-uqsl2}, we obtain
  \begin{equation}
    \label{eq:higher-ind-uqsl2-1}
    \nu_n(u_q) = \nu_n(u_q') = \sum_{a, b = 0}^\infty \, \sum_{\smash{t \in K(a + b)}}
    \left\{ {N - 1 \atop a, n - 1} \right\}_{q^2}
    \left\{ {N - 1 \atop b, n - 1} \right\}_{q^{-2}} q^{2 t (b - a)},
  \end{equation}
  where $K(c) = \{ x \in \mathbb{Z}/N\mathbb{Z} \mid n x \equiv c \pmod{N} \}$. Since, by Lemma~\ref{lem:cong-eq}, $K(c) = \emptyset$ unless $c \equiv 0 \pmod{d}$ ($d = \gcd(N, n)$), we may assume that the first sum of~\eqref{eq:higher-ind-uqsl2-1} is taken over all non-negative integers $a$ and $b$ satisfying
  \begin{equation}
    \label{eq:higher-ind-uqsl2-cong-1}
    a + b \equiv 0 \pmod{d}.
  \end{equation}
  Now assume~\eqref{eq:higher-ind-uqsl2-cong-1} and write $a + b = d \ell$. Then, again by Lemma~\ref{lem:cong-eq},
  \begin{equation*}
    K(a + b) = \{ \check{n} \ell + (N/d) j + N \mathbb{Z} \mid j = 0, \dotsc, d - 1 \}.
  \end{equation*}
  Since the order of $q^{2 (N/d)}$ is $d$, we have
  \begin{align}
    \label{eq:higher-ind-uqsl2-2}
    \sum_{t \in K(a + b)} q^{2 t (b - a)}
    & = \sum_{j = 0}^{d - 1} q^{2 \check{n} \ell (a - b)} \cdot q^{2(a - b)(N/d)j} \\
    & = q^{2 \check{n} \ell (b - a)} \times
    \begin{cases}
      d & \text{if $b - a \equiv 0 \pmod{d}$}, \\
      0 & \text{otherwise}.
    \end{cases} \notag
  \end{align}
  Hence, in addition to~\eqref{eq:higher-ind-uqsl2-cong-1}, we may assume
  \begin{equation}
    \label{eq:higher-ind-uqsl2-cong-2}
    b - a \equiv 0 \pmod{d}.
  \end{equation}
  Since $d$ is odd, \eqref{eq:higher-ind-uqsl2-cong-1} and \eqref{eq:higher-ind-uqsl2-cong-2} imply $a \equiv b \equiv 0 \pmod{d}$. Following, we rewrite $a$ and $b$ in \eqref{eq:higher-ind-uqsl2-1} as $d i$ and $d j$ ($i, j \ge 0$), respectively, and then obtain
  \begin{equation}
    \label{eq:higher-ind-uqsl2-3}
    \nu_n(u_q(\mathfrak{sl}_2))
    = d \sum_{i, j = 0}^\infty
    \left\{ {N - 1 \atop d i, n - 1} \right\}_{q^2} \left\{ {N - 1 \atop d j, n - 1} \right\}_{q^{-2}}
    q^{2 \check{n} d (i + j) (j - i)}
  \end{equation}
  by~\eqref{eq:higher-ind-uqsl2-2}. The proof is done by comparing \eqref{eq:higher-ind-uqsl2-3} with \eqref{eq:higher-ind-Taft-1}.
\end{proof}

\begin{remark}
  Theorem~\ref{thm:higher-ind-uqsl2} can be thought as Theorem~\ref{cor:2nd-ind-factor-uq} with $\mathfrak{g} = \mathfrak{sl}_2$ but with $n$ arbitrary. It is interesting to find such a formula for $\nu_n(u_q(\mathfrak{g}))$ for general semisimple Lie algebras $\mathfrak{g}$.
\end{remark}

\subsection{Explicit values of the indicators}
\label{sec:expl-valu-indic}

Theorem~\ref{thm:higher-ind-Taft} reduces the computation of the indicators of $T_{N^2}(\omega)$ to the evaluation of \eqref{eq:q-function} at $L = N - 1$ and $q = \omega$. Here we consider the case where $N = 2, 3$.

First we assume $N = 2$. Let $a$ and $m$ be integers satisfying $0 \le a \le m$. Then the summand of the right-hand side of~\eqref{eq:q-function} is zero unless $j_1 = \dotsc = j_a = 1$ and $j_{a + 1} = \dotsc = j_m = 0$. Hence we have
\begin{equation*}
  \left\{ {1 \atop a, m } \right\}_{-1} = (-1)^a
\end{equation*}
for $0 \le a \le m$. By using this result, we compute
\begin{equation*}
  \nu_1(T) = 1, \quad \nu_2(T) = 2, \quad \nu_3(T) = 3 \text{\quad and \quad} \nu_4(T) = 4,
\end{equation*}
where we have abbreviated $T = T_4(-1)$. By Lemma \ref{lem:min-pol-dual-Ch}, the minimal polynomial $\phi_T(X)$ divides $(X^2 - 1)^2$. Thus, solving the recurrence equation $x_{n + 4} - 2 x_{n + 2} + x_{n} = 0$ with the initial values $x_n = n$ ($1 \le n \le 4$), we obtain:

\begin{proposition}
  [Example~3.4 of {\cite{KMN09}}]
  $\nu_n(T_4(-1)) = n$ for all $n \ge 1$.
\end{proposition}

The above argument is summarized and generalized as follows: First, for each $n = 1, \dotsc, N^2$, compute the $n$-th indicator of $T_{N^2}(\omega)$ in some way, by hand, or by using a computer. Then solve the linear recurrence equation
\begin{equation*}
  x_{n} - \binom{N}{1} x_{n + N} + \binom{N}{2} x_{n + 2 N} - \dotsb + (-1)^N \binom{N}{N} x_{n + N^2} = 0
  \quad (n \ge 1)
\end{equation*}
with the initial values $x_i = \nu_i(T_{N^2}(\omega))$ for $i = 1, \dotsc, N^2$. As a result, we get
\begin{equation*}
  \nu_n(T_{N^2}(\omega)) = x_n \quad (n = N^2 + 1, N^2 + 2, \dotsc).
\end{equation*}
Following this scheme, let us compute the indicators of $T_9(\omega)$. The following result has been noted in \cite[p.72]{KMN09} without details (as we have mentioned in \S\ref{sec:ex-pointed}, their $T_{N^2}(\omega)$ is slight different from ours).

\begin{proposition}
  \label{prop:Taft-9}
  For all $n \ge 1$, we have
  \begin{equation}
    \label{eq:Taft-9-ind}
    \nu_n(T_9(\omega)) = n \times \begin{cases}
      1 - \omega   & \text{if $n \equiv 0$}, \\
      1            & \text{if $n \equiv 1$}, \\
      -\omega & \text{if $n \equiv 2 \pmod{3}$}.
    \end{cases}
  \end{equation}
\end{proposition}
\begin{proof}
  Fix $a, m \in \mathbb{Z}$ such that $0 \le a \le 2 m$. First we compute
  \begin{equation}
    \label{eq:prop-Taft-9-1}
    \left\{ { 2 \atop a, m } \right\}_\omega
    = \sum_{j_1 + \dotsb + j_m = a} \omega^{j_1^2 + \dotsb + j_m^2}
    \binom{2}{j_1}_{\!\omega} \binom{j_1}{j_2}_{\!\omega} \dotsb \binom{j_{m-1}}{j_m}_{\!\omega}.
  \end{equation}
  Let $(j_1, \dotsc, j_m)$ be an $m$-tuple of integers satisfying $j_1 + \dotsb + j_m = a$ and \eqref{eq:partition} with $L = 2$. Then we have $j_1 = \dotsc = j_b = 2$, $j_{b + 1} = \dotsb = j_{b + c} = 1$ and $j_{b + c + 1} = \dotsc = j_m = 0$ for some $b, c \in \mathbb{Z}$ such that
  \begin{equation}
    \label{eq:prop-Taft-9-2}
    b, c \ge 0, \quad b + c \le m, \text{\quad and \quad} 2 b + c = a.
  \end{equation}
  Put $\ell = b + c$. Then we have $b = a - \ell$, $c = 2 \ell - a$ and
  \begin{equation*}
     \omega^{j_1^2 + \dotsb + j_m^2}
     \binom{2}{j_1}_{\!\omega} \binom{j_1}{j_2}_{\!\omega} \dotsb \binom{j_{m-1}}{j_m}_{\!\omega}
     = \omega^{\ell} (1 + \big[ a \ne 2 \ell \big] \cdot \omega),
  \end{equation*}
  where, given a statement $P$, we have used the Iverson bracket $[P] \in \{ 0, 1 \}$ to indicate one if $P$ is true and zero otherwise. For simplicity of notation, put $\ell_0 = \lceil a / 2 \rceil$ and $\ell_1 = \min \left\{ a, m \right\}$. By~\eqref{eq:prop-Taft-9-2}, $\ell_0 \le b + c \le \ell_1$. Rewriting \eqref{eq:prop-Taft-9-1} as the sum taken over $\ell = b + c$, we obtain
  \begin{equation*}
    \left\{ { 2 \atop a, m } \right\}_\omega
    = \sum_{\ell = \ell_0}^{\ell_1} \omega^{\ell} (1 + \big[ a \ne 2 \ell \big] \cdot \omega)
    = \sum_{\ell = \ell_0}^{\ell_1} \omega^{\ell}
    + \big[ a \in 2 \mathbb{Z} \big] \cdot \omega^{a/2 + 1}.
  \end{equation*}
  By a case-by-case analysis depending on $a \in \mathbb{Z}/6\mathbb{Z}$, $m \in \mathbb{Z}/3\mathbb{Z}$ and whether $a < m$ or not, we obtain Table \ref{tab:q-function}.

  Now it is straightforward to check \eqref{eq:Taft-9-ind} for $1 \le n \le 9$. By solving the linear recurrence equation $x_{n + 9} - 6 x_{n + 6} + 6 x_{n + 3} - x_{n} = 0$, we see that~\eqref{eq:Taft-9-ind} holds for all $n \ge 1$.
\end{proof}

\begin{table}
  \centering
  \begin{minipage}{.49\linewidth}
    \centering
    \begin{tabular}{c|ccc}
      & $m \equiv 0$ & $m \equiv 1$ & $m \equiv 2$ \\ \hline
      $a \equiv 0$ &  $1$ &  $1$ &  $1$ \\
      $a \equiv 1$ & $-1$ & $-1$ & $-1$ \\
      $a \equiv 2$ &  $0$ &  $0$ &  $0$ \\
      $a \equiv 3$ &  $1$ &  $1$ &  $1$ \\
      $a \equiv 4$ & $-1$ & $-1$ & $-1$ \\
      $a \equiv 5$ &  $0$ &  $0$ &  $0$
    \end{tabular}

    \vspace{.5\baselineskip}
    (i) The case where $a < m$.
  \end{minipage}
  \begin{minipage}{.49\linewidth}
    \centering
    \begin{tabular}{c|ccc}
      & $m \equiv 0$ & $m \equiv 1$ & $m \equiv 2$ \\ \hline
      $a \equiv 0$ & $1$ & $0$ & $-\omega$ \\
      $a \equiv 1$ & $0$ & $-1$ & $-\omega^2$ \\
      $a \equiv 2$ & $-\omega^2$ & $\omega$ & $0$ \\
      $a \equiv 3$ & $1$ & $0$ & $-\omega$ \\
      $a \equiv 4$ & $0$ & $-1$ & $-\omega^2$ \\
      $a \equiv 5$ & $-\omega^2$ & $\omega$ & $0$
    \end{tabular}

    \vspace{.5\baselineskip}
    (ii) The case where $a \ge m$.
  \end{minipage}

  \vspace{.5\baselineskip}
  \caption{The value of $\left\{{ 2 \atop a, m }\right\}_\omega$ at $\omega^3 = 1$ for $0 \le a \le 2 m$}
  \label{tab:q-function}
\end{table}

\begin{proposition}
  \label{prop:uqsl2-27}
  Let $q$ be a primitive third root of unity. Then, for all $n \ge 1$, the $n$-th indicator of $u_q := u_q(\mathfrak{sl}_2)$ is given by $\nu_n(u_q) = n^2$
\end{proposition}

This follows from Theorem~\ref{thm:higher-ind-uqsl2} and Proposition~\ref{prop:Taft-9}.

In particular, $\phi_{u_q}(X) = (X-1)^2$, $e(u_q) = 1$ and $m(u_q) = 2$.

\subsection{Remarks}

In \S\ref{subsec:indicator-Taft}, we have derived \eqref{eq:higher-ind-Taft-1} in a quite direct way. In the preceding \S\ref{subsec:indicator-uqsl2}, we have derived \eqref{eq:higher-ind-uqsl2-3} with the help of Theorem~\ref{thm:ind-filtered}. Comparing \eqref{eq:higher-ind-uqsl2-3} with \eqref{eq:higher-ind-Taft-1}, we have obtained Theorem~\ref{thm:higher-ind-uqsl2}. These results reduce the computation of the indicators to the evaluation of \eqref{eq:q-function} at $q$ being a root of unity. At this time, we do not know any effective way to compute \eqref{eq:q-function}. However, since it is the generating function of certain type of partitions, we would expect that a combinatorial approach would be helpful for this problem.

In \S\ref{sec:expl-valu-indic}, we have determined the indicators of $T_{N^2}(\omega)$ for $N = 2, 3$. For these values of $N$, there exists a sequence $c_{\omega} = \{ c_{\omega}(n) \}_{n \ge 1}$ such that
\begin{equation}
  \label{eq:higher-ind-conj-Taft}
  \nu_n(T_{N^2}(\omega)) = c_{\omega}(n) \cdot n,
  \quad c_{\omega}(N + n) = c_{\omega}(n)
  \text{\quad and \quad} c_{\omega}(n) \in \mathbb{Z}[\omega]
\end{equation}
for all $n \ge 1$. By using a computer, one can continue the computation and observe that there seems to be such a sequence for each $N$. For example, for $N = 4, 5, 6, 7$, such a sequence exists and is given by Table~\ref{tab:c-omega} (in the table, `$\equiv$' means congruence modulo $N$). Note that if~\eqref{eq:higher-ind-conj-Taft} would hold, then
\begin{equation}
  \label{eq:higher-ind-conj-uqsl2}
  \nu_n(u_q(\mathfrak{sl}_2)) = \frac{|c_{q^2}(n)|^2 n^2}{\gcd(\ord(q), n)}.
\end{equation}

\begin{table}
  \centering
  \begin{minipage}{.44\linewidth}
    \centering
    \begin{tabular}[b]{c|c}
      $n$ & $c_{\omega}(n)$ \\ \hline
      $n \equiv 0$ & $2 \cdot(1 - \omega)$ \\
      $n \equiv 1$ & $1$ \\
      $n \equiv 2$ & $1 - \omega$ \\
      $n \equiv 3$ & $-\omega$ \\
      \multicolumn{2}{c}{}
    \end{tabular}

    \vspace{.5\baselineskip}
    (i) The case where $N = 4$.
  \end{minipage}
  \begin{minipage}{.55\linewidth}
    \centering
    \begin{tabular}{c|c}
      $n$ & $c_{\omega}(n)$ \\ \hline
      $n \equiv 0$ & $5 \cdot (1 - \omega^{-1})^{-1}$ \\
      $n \equiv 1$ & $1$ \\
      $n \equiv 2$ & $-\omega-\omega^2$ \\
      $n \equiv 3$ & $-\omega-\omega^2-\omega^3$ \\
      $n \equiv 4$ & $-\omega$ \\
    \end{tabular}

    \vspace{.5\baselineskip}
    (ii) The case where $N = 5$.
  \end{minipage}

  \vspace{\baselineskip}
  \begin{minipage}{.44\linewidth}
    \centering
    \begin{tabular}{c|c}
      $n$ & $c_{\omega}(n)$ \\ \hline
      $n \equiv 0$ & $6 \cdot (1 - \omega)$ \\
      $n \equiv 1$ & $1$ \\
      $n \equiv 2$ & $2 \cdot (1 - \omega)$ \\
      $n \equiv 3$ & $3 \cdot (1 - \omega)$ \\
      $n \equiv 4$ & $2 \cdot (1 - \omega)$ \\
      $n \equiv 5$ & $-\omega$ \\
      \multicolumn{2}{c}{}
    \end{tabular}

    \vspace{.5\baselineskip}
    (iii) The case where $N = 6$.
  \end{minipage}
  \vspace{2\baselineskip}
  \begin{minipage}{.55\linewidth}
    \centering
    \begin{tabular}{c|c}
      $n$ & $c_{\omega}(n)$ \\ \hline
      $n \equiv 0$ & $7 \cdot (1-\omega^{-1})^{-1}$ \\
      $n \equiv 1$ & $1$ \\
      $n \equiv 2$ & $-\omega-\omega^2-\omega^3$ \\
      $n \equiv 3$ & $-\omega-\omega^2$ \\
      $n \equiv 4$ & $-\omega-\omega^2-\omega^3-\omega^4-\omega^5$ \\
      $n \equiv 5$ & $-\omega-\omega^2-\omega^3-\omega^4$ \\
      $n \equiv 6$ & $-\omega$ \\
    \end{tabular}

    \vspace{.5\baselineskip}
    (iv) The case where $N = 7$.
  \end{minipage}

  \vspace{-.5\baselineskip}
  \caption{$c_{\omega}$ for $N = 4, 5, 6, 7$}
  \label{tab:c-omega}
\end{table}

\subsubsection*{Frobenius theorem}

The Frobenius theorem in finite group theory states that \eqref{eq:ind-group-algebra} is divisible by $\gcd(n, |G|)$. Since~\eqref{eq:ind-group-algebra} is the $n$-th indicator of the group algebra $k G$, it is natural to ask the following question:

\begin{question}
  Is $\nu_n(H)/\gcd(n, \dim_k(H))$ an algebraic integer? 
\end{question}

This question has been answered positively for several families of semisimple Hopf algebras by the author \cite{MR2774703,KenichiShimizu:2012}. However, the general answer is not known even in the semisimple case. As observed in the above, if $H$ is the Taft algebra, then it is likely that $\nu_n(H) / n$ is in $\mathbb{Z}[\omega]$. More interestingly, $\nu_n(u_q(\mathfrak{sl}_2))$ seems to be divisible by $n^2$. In the case where $\ord(q) = 3$, this has been verified in Proposition \ref{prop:uqsl2-27}. If $\ord(q) = 5$, then
\begin{equation*}
  \frac{\nu_n(u_q(\mathfrak{sl}_2))}{n^2} = \begin{cases}
    - 3 (q + q^{-1}) - 2 (q^2 + q^{-2}) & n \equiv 0 \\
    1 & n \equiv 1, 4, \\
    - 2 (q + q^{-1}) - (q^2 + q^{-2}) & n \equiv 2,3 \pmod{5} \\
  \end{cases}
\end{equation*}
by~\eqref{eq:higher-ind-conj-uqsl2} and Table~\ref{tab:c-omega} (ii). In a similar way, if $\ord(q) = 7$, then
\begin{equation*}
  \frac{\nu_n(u_q(\mathfrak{sl}_2))}{n^2} = \begin{cases}
    -6 (q+q^{-1}) -3 (q^2+q^{-2}) -5(q^3+q^{-3}) & n \equiv 0 \\ 
    1                                          & n \equiv 1,6, \\
    -2 (q+q^{-1}) -  (q^2+q^{-2}) -2(q^3+q^{-3}) & n \equiv 2,5, \\
    -3 (q+q^{-1}) -  (q^2+q^{-2}) -2(q^3+q^{-3}) & n \equiv 3,4 \pmod{7}.
  \end{cases}
\end{equation*}


\end{document}